\documentclass{compositio}

\usepackage{graphicx}
\usepackage{amssymb, amsmath, amsthm, enumerate, mathrsfs, tikz, mathtools, wasysym, pifont}
\usepackage{hyperref}
\usetikzlibrary{calc, decorations.markings}
\usetikzlibrary{positioning}
\usetikzlibrary{matrix}
\usetikzlibrary{cd}
\date{}

\renewcommand\Re{\operatorname{\mathfrak{Re}}}
\renewcommand\Im{\operatorname{\mathfrak{Im}}}

\newcommand{\eps}{\varepsilon}

\newcommand{\spac}{\operatorname{span}_{\mathbb{C}}}
\newcommand{\id}{\operatorname{id}}

\newcommand{\tr}{\operatorname{tr}}
\newcommand{\Hom}{\operatorname{Hom}}

\newcommand{\ph}{\operatorname{\varphi}}
\newcommand{\mfrak}{\mathfrak}
\newcommand{\bslsh}{\backslash}

\newcommand*{\myfont}{\fontfamily{pag}\selectfont}
\newcommand\Sx{\text{{\myfont X}}}
\newcommand\So{\text{{\myfont O}}}

\newcommand{\AAQ}{\operatorname{\mathbb{A}_{\mathbb{Q}}}}

\newcommand{\sgn}{\operatorname{sgn}}
\newcommand{\Aut}{\operatorname{Aut}}

\newcommand{\Tr}{\operatorname{Tr}}
\newcommand{\diag}{\operatorname{diag}}

\newcommand{\pr}{\operatorname{pr}}
\newcommand{\Vol}{\operatorname{Vol}}

\newtheorem{lemma}{Lemma}[section]
\newtheorem{theorem}[lemma]{Theorem}
\newtheorem{corollary}[lemma]{Corollary}
\newtheorem{proposition}[lemma]{Proposition}

\theoremstyle{definition}
\newtheorem{definition}[lemma]{Definition}

\newtheorem{remark}[lemma]{Remark}

\numberwithin{equation}{section}

\begin{document}
\flushbottom

\title{Algebraicity of metaplectic $L$-functions}

\author{Salvatore Mercuri}
\email{smmercuri@gmail.com}
\address{Department of Mathematical Sciences,
	Durham University, \newline
	Stockton Road,\newline
	Durham DH1 3LE,\newline
	UK.}

\thanks{Completed with the support of an EPSRC Doctoral Studentship (Grant No. 000118421)}

\begin{abstract}
Notable results on the special values of $L$-functions of Siegel modular 
forms were obtained by J. Sturm in the case when the degree $n$ is even 
and the weight $k$ is an integer. In this paper we extend this method to
half-integer weights $k$ and arbitrary degree $n$, 
determining the algebraic field in which they lie. This method hinges on the Rankin-Selberg method; our extension of this is aided by the theory of 
half-integral modular forms developed by G. Shimura. In the second half,  
an analogue of P. B. Garrett's conjecture is proved 
in this setting, a result that is of independent interest but that bears direct 
applications to our first results. It determines exactly how the decomposition of modular forms into cusp forms and Eisenstein series preserves algebraicity and, ultimately, the full range of special values.
\end{abstract}

\maketitle

\section{Introduction} 
\label{intro}

\noindent In the philosophy of the oft-touted Langlands program $L$-functions 
associated to motives on one side are paired off with $L$-functions attached to 
automorphic forms on the other. When it comes to Siegel modular forms of integral 
weight $k$ and arbitrary degree $n$, through Deligne's conjecture results on the 
motivic side are known to such an extent that the special values of their automorphic 
$L$-functions are at this point expected. However, in the case that $k$ is a half-integral 
weight, the corresponding motives are not even known to exist, so that results 
on special values of automorphic $L$-functions are in this case less expected.

Of primary interest in this paper are $L$-functions that we associate to Siegel 
modular forms of half-integral weight -- sobriquet \emph{metaplectic modular forms} -- and their special values, with the aim 
here being a precise determination of the number field in which these values lie. 
Prior work by Shimura in \cite{Shimurabook} established that these values 
belonged to the algebraic closure $\overline{\mathbb{Q}}$, with the Siegel modular 
forms being defined over an arbitrary totally real number field. In \cite{Bouganis}, Bouganis works further to specify 
the exact algebraic number field in which they lie, which kind of precision this 
paper also pursues, but in so doing some additional conditions on the characters 
of the modular forms were required (Theorem 6.2 (i)-(ii) of \cite{Bouganis}) which we avoid. For concreteness, we 
limit ourselves to Siegel modular forms defined over $\mathbb{Q}$, though it is believed that these results would easily generalise at least to 
the case of totally real number fields of class number one. 

This is a paper of two halves. In the first, Sections \ref{modform} -- \ref{algebraicitysection}, the method employed by Sturm in \cite{Sturm} is extended to the present case which is facilitated (the extension) by the integral expression (4.1) of Shimura's paper \cite{Shimuraexp}. This uses the Rankin-Selberg method to express the $L$-function of an eigenform $f$ as an integral $\langle f, \theta\mathcal{E}\rangle$ of $f$ against a theta series multiplied by a \emph{non-holomorphic} Eisenstein series. Holomorphic projection is applied to this latter form to produce a \emph{holomorphic} cusp form $K$ and a subsequent expression of the $L$-function in terms of $\langle f, K\rangle$; this projection is given in Sect. \ref{Hproj} and the altered integral expression in Sect. \ref{integralsection}. The special values determined  by this method are precisely those values of the Eisenstein series at which holomorphic projection to a \emph{cusp form} is applicable. Finally, algebraicity of quotients of the form
\[
	\langle f, g\rangle\mu(f)^{-1},
\]
where $g$ is a holomorphic form and $\mu(f)$ is a non-zero constant dependent only on the eigenclass of $f$, is proved. This yields algebraicity of these special values. 

Though the algebraicity of the values in the first half is very strong, the actual set of values produced is not optimal; it is smaller than the full range given in Theorem 28.8 of \cite{Shimurabook}. The set of values has been constrained by the requirement that holomorphic projection produce a \emph{cusp form} the removal of which is the focus of the second half, Sections \ref{eisensteinsection}--\ref{specialvalues}. Most of this is taken up by a proof of a particular case of Paul B. Garrett's conjecture that if $f$ has algebraic coefficients then its Klingen Eisenstein series $E(f)$ does too, \cite{Garrett}. This is done by a non-trivial extension of Harris' method found in \cite{Harris}, whose setting is integral weight and full level Siegel modular forms, and which was futher extended to more general automorphic forms that are associated to Shimura varieties in \cite{Harristwo}. Metaplectic modular forms do not associate to Shimura varieties, hence the non-triviality and, in Sect. \ref{eisensteinsection}, a precise field extension of $\mathbb{Q}$ is given through which the conjecture holds. The relevant corollary of this work is the means through which to specify an extension $\mathscr{L}/\mathbb{Q}$ whereby the well-known decomposition $\mathcal{M}_k = \mathcal{S}_k\oplus\mathcal{E}_k$ preserves algebraicity of the Fourier coefficients of the forms involved; we write this as
\begin{align} 
	\mathcal{M}_k(\mathscr{L}) = \mathcal{S}_k(\mathscr{L})\oplus\mathcal{E}_k(\mathscr{L}). \label{introdecomp}
\end{align}
Such a decomposition was already shown by Shimura in \cite{Shimurabook} when $\mathscr{L} = \overline{\mathbb{Q}}$. As has been alluded to supra, such a decomposition allows the determination of the full set of special values by stipulating that the projection of $\theta\mathcal{E}$ need only be a holomorphic modular form. In splitting up $K = K_{\mathcal{S}} + K_{\mathcal{E}}$ per the decomposition (\ref{introdecomp}) we are left with $\langle f, K_{\mathcal{S}}\rangle$ by orthogonality. The methods of the first half now follow giving the full set of special values, but with the slightly weaker algebraicity caused by the addition of $\mathscr{L}$.

\section{Modular forms of half-integral weight}
\label{modform}

To begin with, we run through some general groundwork and notation. Let $\AAQ$ and $\mathbb{I}_{\mathbb{Q}}$ denote the adele ring 
and idele group, respectively, of $\mathbb{Q}$. The set of Archimedean places is denoted by $\infty$ and 
the non-Archimedean places by $\mathbf{f}$. For any fractional ideal 
$\mfrak{r}$ of $\mathbb{Q}$ let $\mfrak{r}_p$ denote the completion (with respect to the $p$-adic absolute value) of the localisation of $\mfrak{r}$ at the prime $p$, which is an ideal of $\mathbb{Z}_p$, and understand $N(\mfrak{r})\in\mathbb{Q}_{\geq 0}$ to be the unique positive generator of $\mfrak{r}$. For any element $t\in\mathbb{I}_F$ we denote by $t\mfrak{r}$ 
the fractional ideal of $\mathbb{Q}$ such that $(t\mfrak{r})_p=t_p\mfrak{r}_p$ for any 
$p\in\mathbf{f}$. We recall the adelic norm 
\[ 
|t|_{\mathbb{A}} = \prod_v|t_v|_v,
\]
where the valuations $|\cdot|_v$ are normalised. Let $\mathbb{T}$ denote the unit 
circle; define three characters on $\mathbb{C}$, $\mathbb{Q}_p$, 
and $\mathbb{A}_{\mathbb{Q}}$ respectively, with images in $\mathbb{T}$, by
\begin{align*}
	e:z&\mapsto e^{2\pi iz}, \\
	e_p:x&\mapsto e(-\{x\}), \\
	e_{\mathbb{A}}:x&\mapsto e(x_{\infty})
\prod_{p\in\mathbf{f}}e_p(x_p),
\end{align*}
where $\{x\}$ denotes the fractional part of $x\in\mathbb{Q}_p$.
If $x\in \AAQ$ then we also put $e_{\mathbf{f}}(x)=e_{\mathbb{A}}(x_{\mathbf{f}})$
and $e_{\infty}(x)=e\left(x_{\infty}\right)$. For any matrix $x\in M_n(\mathbb{C})$ write 
$x>0$ ($x\geq 0$) to mean that $x$ is positive definite (respectively positive 
semi-definite); $|x|=\det(x)$ and $\|x\|=|\det(x)|$; and $\tilde{x} = (x^T)^{-1}$. For any collection of matrices $x_1, \dots, x_r$, let $\diag[x_1, \dots, x_r]$ denote the matrix whose $i$th diagonal block is $x_i$, for $1\leq i\leq r$, and is zero everywhere else. If $\alpha\in GL_{2n}(F)$ then 
put
\[
	\alpha=\begin{pmatrix} a_{\alpha} & b_{\alpha} \\ c_{\alpha} & d_{\alpha}\end{pmatrix},
\]
where $a_{\alpha}, b_{\alpha}, c_{\alpha}, d_{\alpha}\in M_n(F)$. We define an 
algebraic group $G$, subgroups $P, \Omega\leq G$, and the generalised upper 
half-plane $\mathbb{H}_n$ by
\begin{align*}
	G&:=Sp_n(\mathbb{Q})=\{\alpha\in GL_{2n}(\mathbb{Q})\mid \alpha^T\iota\alpha=\iota\}, \hspace{20pt} 
\iota:=\begin{pmatrix} 0 & -I_n \\ I_n & 0\end{pmatrix}, \\
	P&:=\{\alpha\in G\mid c_{\alpha}=0\}, \\
	\Omega&:=\{\alpha\in G_{\mathbb{A}}\mid \det(c_{\alpha})\in \mathbb{I}_{\mathbb{Q}}\}, \\
	\mathbb{H}_n&:=\{z=x+iy\in M_n(\mathbb{C})\mid z^T=z, y>0\},
\end{align*}
where $G_{\mathbb{A}} := Sp_n(\AAQ)$ denotes the adelization of $G=Sp_n(\mathbb{Q})$.

A half-integral weight is an element $k\in\mathbb{Q}$ such 
that $k-\frac{1}{2}\in\mathbb{Z}$. 
The factor of automorphy of a half-integral weight will involve taking a square root, and
to choose such a root satisfactorily we make use of the metaplectic group. This is understood
as the double cover of $Sp_n$ and is denoted $Mp_n$. Set $M_p := Mp_n(\mathbb{Q}_p)$ for all
$p$ and $M_{\mathbb{A}}$ to be the adelization. 

We have natural projections $\pr_{\mathbb{A}}:M_{\mathbb{A}}\to G_{\mathbb{A}}$ 
and $\pr_p:M_p\to G_p$ which are both denoted by $\pr$ when the context is clear. There is a natural lift 
$r:G\to M_{\mathbb{A}}$ through which we can and do view $G$ as a subgroup of 
$M_{\mathbb{A}}$. There exist further lifts $r_P:P_{\mathbb{A}}\to M_{\mathbb{A}}$ 
and $r_{\Omega}:\Omega\to M_{\mathbb{A}}$ which are equal to $r$ on $P$ and 
$G\cap \Omega$ respectively, and such that
\[ 
	r_{\Omega}(\alpha\beta\gamma)=r_P(\alpha)r_{\Omega}(\beta)r_P(\gamma) 
\]
for $\alpha, \beta\in P_{\mathbb{A}}$ and $\beta\in\Omega$, \cite[p. 24]{Shimurahalf}.

Recall that there is a natural action of $Sp_n(\mathbb{R})$ on 
$\mathbb{H}_n$ given by 
\[ 
	\gamma\cdot w:=(a_{\gamma}z+b_{\gamma})(c_{\gamma}z+d_{\gamma})^{-1} 
\]
for $\gamma\in Sp_n(\mathbb{R}), z\in\mathbb{H}_n$, and further define 
\begin{align*}
	\Delta(z)&:=|\Im(z)|, \\
	j(\gamma, z)&:=|c_{\gamma}z+d_{\gamma}|.
\end{align*}
If, now, $\alpha\in G_{\mathbb{A}}$ then 
$\alpha_{\infty}\in Sp_n(\mathbb{R})$, so we naturally extend 
the above
\begin{align*}
	\alpha\cdot z&:=\alpha_{\infty}\cdot z, \\
	\Delta(z)&:=\Delta(z_{\infty}), \\
	j(\alpha, z)&:=j(\alpha_{\infty}, z_{\infty}).
\end{align*}

For any two fractional ideals $\mfrak{x}, \mfrak{y}$ of $\mathbb{Q}$ such that 
$\mfrak{xy}\subseteq\mathbb{Z}$, congruence subgroups are defined by the following subsets of $G_p$, $G_{\mathbb{A}}$, and $G$ 
respectively, by
\begin{align*}
	D_p[\mfrak{x}, \mfrak{y}]&:=\{x\in G_p\mid a_x, d_x\in M_n(\mathbb{Z}), b_x\in 
M_n(\mfrak{x}_p), c_x\in M_n(\mfrak{y}_p)\}, \\
	D[\mfrak{x}, \mfrak{y}]&:=Sp_n(\mathbb{R})\prod_{p}
D_p[\mfrak{x}, \mfrak{n}],\\
	\Gamma[\mfrak{x}, \mfrak{y}] &:= G\cap D[\mfrak{x}, \mfrak{y}].
\end{align*}
Typically these will take the form $\Gamma[\mfrak{b}^{-1}, \mfrak{bc}]$ for certain fractional ideals $\mfrak{b}$ and integral ideals $\mfrak{c}$.

Take a half-integral weight $k$ and put $[k]:=k-\frac{1}{2}\in\mathbb{Z}$; if $\ell\in\mathbb{Z}$ then $[\ell] := \ell$.
Our factor of automorphy for modular forms of half-integral weight will come in 
two parts, one of which is the familiar factor of weight $[k]$ and the 
other acts as a factor of automorphy of weight $\frac{1}{2}$. 
The major caveat in this setting is that the factor of weight $\frac{1}{2}$ is only 
definable for a particular subset $\mfrak{M}\subseteq M_{\mathbb{A}}$ given by
\begin{align*}
	C_p^{\theta}&=\{\xi\in D_p[1, 1]\mid (\alpha_{\xi}b_{\xi}^T)_{ii}
\in 2\mathbb{Z}_p, (c_{\xi}d_{\xi}^T)_{ii}\in 2\mathbb{Z}_p, 1\leq i\leq n\},  \\
	C^{\theta}&=Sp_n(\mathbb{R})\prod_{p}C_p^{\theta}, \\
	\mfrak{M}&=\{\sigma\in M_{\mathbb{A}}\mid \alpha=\pr(\sigma)\in P_{\mathbb{A}}C^{\theta}\}.
\end{align*}
So in considering modular forms of half-integral weight, we must ensure that $D[\mfrak{b}^{-1}, \mfrak{bc}]\subseteq\mfrak{M}$.

For any $\sigma\in M_{\mathbb{A}}$ we set $x_{\sigma} = x_{\alpha}$ where $x\in\{a, b, c, d\}$ and $\alpha=\pr(\sigma)\in G_{\mathbb{A}}$; define $\sigma\cdot z=\alpha\cdot z$ for $z\in\mathbb{H}_n$. If $\sigma\in\mfrak{M}$ then we may define a holomorphic function 
$h_{\sigma}=h(\sigma, \cdot):\mathbb{H}_n\to\mathbb{C}$ 
satisfying the following properties, the proofs for which we refer the reader to 
\cite[pp. 294--295]{Shimuraold}:
\begin{align}
	&h(\sigma, z)^2 = \zeta j(\pr(\sigma), z)\ \text{for a constant 
$\zeta\in\mathbb{T}$; $h(\sigma, z)\in\mathbb{T}$ if $\pr(\sigma)_{\infty}=I_{2n}$}; \label{h1} \\
	&h(\rho\sigma\tau, z) = h(\rho, z)h(\sigma, \tau z)h(\tau, z)\
\text{if $\pr(\rho)\in P_{\mathbb{A}}, \pr(\tau)\in C^{\theta}$}.  \label{h3}
\end{align}
The factor of automorphy for half-integral weights $k$ is then given as
\[ 
	j_{\sigma}^k(z)=h_{\sigma}(z)j(\alpha, z)^{[k]},
\]
where $\sigma\in\mfrak{M}, \alpha=\pr(\sigma)\in G_{\mathbb{A}},$ and $z\in\mathbb{H}_n$. If $\ell\in\mathbb{Z}$ then the factor of automorphy is defined as usual
\[
	j_{\alpha}^{\ell}(z) : =j(\alpha, z)^{\ell},
\]
where $\alpha\in G_{\mathbb{A}}$ and $z\in\mathbb{H}_n$.

Given $\kappa\in\frac{1}{2}\mathbb{Z}$ (integral or half-integral), a function $f:\mathbb{H}_n\to\mathbb{C}$, and a $\xi\in G_{\mathbb{A}}$ or $\mfrak{M}$ according as $\kappa\in\mathbb{Z}$ nor not, we define the \emph{slash operator} as 
\[
	(f||_{\kappa}\xi)(z) = j_{\xi}^{\kappa}(z)^{-1}f(\xi\cdot z),
\]
for $z\in\mathbb{H}_n$. If $\Gamma\leq G$ is a congruence subgroup such that $\Gamma\leq\mfrak{M}$, then let $C_{\kappa}^{\infty}(\Gamma)$ denote the set of analytic functions $\mathbb{H}_n\to\mathbb{C}$ 
that satisfy $f||_{\kappa}\xi=f$ for any $\xi\in\Gamma$. Let $\mathcal{M}_{\kappa}(\Gamma)\subseteq C_{\kappa}^{\infty}(\Gamma)$ 
be the subspace of holomorphic functions, $\mathcal{S}_{\kappa}(\Gamma)$ be 
the subspace of cusp forms, and write
\[
	\mathcal{M}_{\kappa}=\bigcup_{\Gamma}\mathcal{M}_{\kappa}(\Gamma),
\hspace{20pt}\mathcal{S}_{\kappa}=\bigcup_{\Gamma}\mathcal{S}_{\kappa}(\Gamma),
\]
where the union is taken over all congruence subgroups of $G$.

The restrictions of a Hecke character $\psi:\mathbb{I}_{\mathbb{Q}}/\mathbb{Q}^{\times}\to\mathbb{T}$ of $\mathbb{Q}$ to $\mathbb{Q}_p^{\times}$, $\mathbb{Q}_{\infty}^{\times}$, and $\mathbb{Q}_{\mathbf{f}}^{\times}$ are denoted $\psi_p, \psi_{\infty}$, and $\psi_{\mathbf{f}}$ respectively. We say $\psi$ is \emph{normalised} if $\psi_{\infty}(x) = \sgn(x_{\infty})^t$ for some $t\in\mathbb{Z}$; we always assume our Hecke characters to be normalised. 

Now take a normalised Hecke character with the following properties:
\begin{align}
	\psi_p(a) &= 1\ \text{if $a\in\mathbb{Z}_p^{\times}$ and $a - 1\in\mfrak{c}_p$}, \label{char1}\\
	 \psi_{\infty}(x)^n &= \sgn(x_{\infty})^{n[\kappa]}. \label{char2}
\end{align}
Assume that $\mfrak{b}^{-1}\subseteq 2\mathbb{Z}$ and $\mfrak{bc}\subseteq 2\mathbb{Z}$, then $\Gamma = \Gamma[\mfrak{b}^{-1}, \mfrak{bc}]\leq\mfrak{M}$. Let
$C_{\kappa}^{\infty}(\Gamma, \psi)$ denote the space of all $F\in C_{\kappa}^{\infty}$ such that
\[
	F||_{\kappa}\gamma=\psi_{\mfrak{c}}(|a_{\gamma}|)F
\]
for all $\gamma\in\Gamma$, where $\psi_{\mfrak{c}}=\prod_{p\mid\mfrak{c}}\psi_p$. Put $\mathcal{M}_{\kappa}(\Gamma, \psi) = C_{\kappa}(\Gamma, \psi)\cap\mathcal{M}_{\kappa}$ and $\mathcal{S}_{\kappa}(\Gamma, \psi) = C_{\kappa}(\Gamma, \psi)\cap\mathcal{S}_{\kappa}$.

Understand $\pr = \id$ if $\kappa\in\mathbb{Z}$. If $f\in\mathcal{M}_{\kappa}(\Gamma, \psi)$ then its adelisation, $f_{\mathbb{A}}:\pr^{-1}(G_{\mathbb{A}}) \to\mathbb{C}$, is a function defined by
\[
	f_{\mathbb{A}}(x) = \psi_{\mfrak{c}}(|d_w|)(f||_{\kappa}w)(\mathbf{i}),
\]
where $x = \alpha w$ for $\alpha\in G$ and $w\in \pr^{-1}(D)$, and $\mathbf{i} = iI_n$.
We have that
\begin{align*}
	f_{\mathbb{A}}(\alpha xw)=\psi_{\mfrak{c}}(|d_w|)j_w^{\kappa}(\mathbf{i})^{-1}f_{\mathbb{A}}(x),
\end{align*}
if $w\cdot\mathbf{i}=\mathbf{i}, \pr(w)\in D[\mfrak{b}^{-1}, \mfrak{bc}]$, and $\alpha\in G$; the above goes conversely (\cite[p. 537]{Shimuraint}).

We define spaces of symmetric matrices as follows
\begin{align*}
	S&:=\{\xi\in M_n(\mathbb{Q})\mid \xi^T=\xi\}, \hspace{20pt} 
&&\hspace{10pt}S_+:=\{\xi\in S\mid \xi\geq 0\}, \\
	S^{\triangledown} &:= \{. \xi\in S\mid \xi_{ii}\in\mathbb{Z}, \xi_{ij}\in\frac{1}{2}\mathbb{Z}, i < j\}, && \hspace{10pt} S_+^{\triangledown} := S^{\triangledown}\cap S_+, \\
	S(\mfrak{r})&:=S\cap M_n(\mfrak{r}), 
&&S_{\mathbf{f}}(\mfrak{r}):=\prod_{p}S(\mfrak{r})_p,
\end{align*}
for any fractional ideal $\mfrak{r}$ of $\mathbb{Q}$. 

\begin{theorem}\label{fourierexpth}[Shimura, \cite{Shimurahalf}, p. 27] Let $\kappa\in\frac{1}{2}\mathbb{Z}$, and put $D = D[\mfrak{b}^{-1}, \mfrak{bc}]$, $\Gamma = \Gamma[\mfrak{b}^{-1}, \mfrak{bc}]$ (assuming that both are contained in $\mfrak{M}$ if $\kappa\notin\mathbb{Z}$). Let $f\in\mathcal{M}_{\kappa}(\Gamma, \psi), q\in GL_n(\mathbb{A}_{\mathbb{Q}})$, and $s\in S_{\mathbb{A}}$, then the adelic Fourier expansion of $f$ is given as
\begin{align*}
	f_{\mathbb{A}}\left(r_P\begin{pmatrix} q & s\tilde{q} \\ 0 & \tilde{q}\end{pmatrix}\right)
= |q_{\infty}|^{[\kappa]}\|q_{\infty}\|^{\kappa-[\kappa]}\sum_{\tau\in S_+}
c(\tau, q; f)e_{\infty}(\tr(\mathbf{i}q^T\tau q))e_{\mathbb{A}}(\tr(\tau s)),
\end{align*}
where $c_f(\tau, q; f) = c_f(\tau, q)\in\mathbb{C}$ and recall $\tilde{q} = (q^T)^{-1}$. Furthermore, the coefficients obey the following:
\begin{enumerate}[(i)]
	\item $c_f(\tau, q)\neq 0$ only if $e_{\mathbb{A}}(\tr(q^T\tau qs))=1$ 
for all $s\in S_{\mathbf{f}}(\mfrak{b}^{-1})$;
	\item $c_f(\tau, q)=c_f(\tau, q_{\mathbf{f}})$;
	\item $c_f(b^T\tau b, q)=|b|^{[\kappa]}\|b\|^{\kappa-[\kappa]}c_f(\tau, bq)$ 
for any $b\in GL_n(\mathbb{Q})$;
	\item $\psi_{\mathbf{f}}(|a|)c_f(\tau, qa)=c_f(\tau, q)$ for any 
$\diag[a, \tilde{a}]\in D[\mfrak{b}^{-1}, \mfrak{bc}]$;
	\item if 
$\beta\in G\cap \diag[r, \tilde{r}]D[\mfrak{b}^{-1}, \mfrak{bc}]$ and $r\in GL_n(\mathbb{A}_F)$, then 
\[
	j_{\beta}^{\kappa}(\beta^{-1}z)f(\beta^{-1}z)=\psi_{\mfrak{c}}(|d_{\beta}r|)
\sum_{\tau\in S_+}c_f(\tau, r)e_{\infty}(\tr(\tau z)).
\]
\end{enumerate}
\end{theorem}
The proof of this expansion and the subsequent properties can be found in Proposition 1.1 of 
\cite{Shimurahalf}. The coefficients $c_f(\tau, 1)$ correspond to the usual Fourier coefficients of $f$. By (i) of the above theorem the coefficients $c_f(\tau, 1)$ of $f$ are zero unless $\tau\in N(\mfrak{b})S^{\triangledown}$.

For any two $f, g\in\mathcal{M}_{\kappa}(\Gamma, \psi)$, where $\kappa\in\frac{1}{2}\mathbb{Z}$, 
define the Petersson inner product
\[
	\langle f, g\rangle=\Vol(\Gamma\bslsh\mathbb{H}_n)^{-1}
\int_{\Gamma\bslsh\mathbb{H}_n}f(z)\overline{g(z)}\Delta(z)^{\kappa}d^{\times}z,
\]
in which
\[
	d^{\times}z=\Delta(z)^{-n-1}\bigwedge_{p\leq q}
(dx_{pq}\wedge dy_{pq}),
\]
for $z=(x_{pq}+iy_{pq})_{p, q=1}^n\in\mathbb{H}_n$.

The elements of $\Aut(\mathbb{C})$ act on the space of modular forms in the usual way. That is, 
if $f\in\mathcal{M}_{\kappa}(\Gamma, \psi)$ 
has Fourier coefficients $c_f(\tau, 1)$ for $\tau\in S_+$ then $f^{\sigma}\in
\mathcal{M}_{\kappa}(\Gamma, \psi^{\sigma})$ is the modular form whose
Fourier coefficients are $c_f(\tau, 1)^{\sigma}$ for all $\tau\in S_+$. 

\section{Holomorphic projection}
\label{Hproj}

Assume that $(\mfrak{b}^{-1}, \mfrak{bc})\subseteq 2\mathbb{Z}\times 2\mathbb{Z}$ 
and put $\Gamma = G\cap D[\mfrak{b}^{-1}, \mfrak{bc}]$. Suppose that $F\in C_k^{\infty}(\Gamma, \psi)$, we say 
that $F$ is of bounded growth if for all $\eps > 0$ we have
\[
	\int_{X}\int_Y|F(z)|\Delta(z)^{k-1-n}e^{-\eps\tr(\Im(z))}dy dx <\infty,
\]
where 
\begin{align*}
	dy&=\bigwedge_{p\leq q}dy_{pq}, \hspace{20pt} dx = \bigwedge_{p\leq q}
dx_{pq}, \hspace{20pt} d^{\times}y=\Delta(z)^{-\frac{n+1}{2}}dy, \\
	Y&=\{y\in M_n(\mathbb{R})\mid y=y^T, y>0\}, \\
	X&=\{x\in M_n(\mathbb{R})\mid x=x^T, |x_{ij}|\leq\tfrac{1}{2}\ \forall\ i, j\}.
\end{align*}
If $F\in C_k^{\infty}(\Gamma)$ then, by Theorem \ref{fourierexpth}, it has an absolutely convergent Fourier 
expansion of the form
\[
	F(z)=\sum_{\tau\in S_{\mfrak{b}}^{\triangledown}} c_F(\tau, y)e(\tr(\tau x)),
\]
where $S_{\mfrak{b}}^{\triangledown}$ is the set of all $\tau\in S$ such that $\tau\in N(\mfrak{b})S^{\triangledown}$, and $c_F(\tau, y)$ are smooth functions of $y$ having values in $\mathbb{C}$. If $f$ is holomorphic then it has a Fourier expansion of the form
\[
	f(z)=\sum_{0\leq\tau\in S_{\mfrak{b}}^{\triangledown}} c_f(\tau, 1)e(\tr(\tau z)).
\]
The following theorem extends to half-integral 
$k$ the notion of holomorphic projection given in Theorem 1 of \cite{Sturm}
when $k$ is integral.

\begin{theorem}\label{holoproj} 
Assume that $k>2n$ and that $F$ 
is of bounded growth. For any $0<\tau\in S_{\mfrak{b}}^{\triangledown}$ set
\begin{align*}
	\mu(k, n)&:=\Gamma_n\left(k-\tfrac{n+1}{2}\right)\pi^{-n\left(k-\frac{n+1}{2}\right)}, \\
 	c(\tau)&:=\mu(k, n)^{-1}\left|4\tau\right|^{k-\frac{n+1}{2}}\int_Y c_F(\tau, y)e^{-2\pi \tr(\tau y)}
|y|^{k-1-n}dy.
\end{align*}
Then define the holomorphic projection map
\begin{align*}
	\mathbf{Pr}:C_k^{\infty}(\Gamma, \psi)&\to \mathcal{S}_k(\Gamma, \psi); \\
	F&\mapsto \sum_{0<\tau\in S_{\mfrak{b}}^{\triangledown}}c(\tau)e(\tr(\tau z)).
\end{align*}
Furthermore, the projection map satisfies $\langle F, g\rangle=\langle \mathbf{Pr}(F), g\rangle$
for any $g\in \mathcal{S}_k(\Gamma', \psi)$ and $\Gamma'\leq\Gamma$ of finite index. 
\end{theorem}

To prove this we introduce the half-integral weight Poincar\'{e} series. Let $\Gamma$ and $k>2n$ 
be as above, fix $0<\tau\in S_{\mfrak{b}}^{\triangledown}$, and let $\Gamma_{\infty}$ be the subgroup of $\Gamma$ generated by $\begin{psmallmatrix} \pm I_n & b \\ 0 & \pm I_n\end{psmallmatrix}$, with $b\in M_n(\mfrak{b}^{-1})$. Then define the single-variable holomorphic Poincar\'{e} series by
\[ 
	G_{\tau}(z):=\sum_{\alpha\in \Gamma_{\infty}\bslsh\Gamma}\psi_{\mfrak{c}}^{-1}(|a_{\alpha}|)
j_{\alpha}^k(z)^{-1}e(\tr(\tau \alpha z)). 
\]
\begin{proposition}\label{poincare} 
\begin{enumerate}[(i)]
	\item The sum defining $G_{\tau}$ converges absolutely and uniformly on compact 
subgroups of $\mathbb{H}_n$ and $G_{\tau}\in \mathcal{S}_k(\Gamma, \psi)$.
	\item If $F\in C_k^{\infty}(\Gamma, \psi)$, then
\[ 
	N(\mfrak{b})^{\frac{n(n+1)}{2}}\Vol(\Gamma\bslsh\mathbb{H}_n)\langle F, G_{\tau}\rangle=\int_Y c_F(\tau, y)
e^{-2\pi \tr(\tau y)}|y|^{k-1-n}dy, 
\]
which integral is absolutely convergent.
	\item If $f\in \mathcal{S}_k(\Gamma, \psi)$ then
\[ 
	N(\mfrak{b})^{\frac{n(n+1)}{2}}\Vol(\Gamma\bslsh\mathbb{H}_n)\langle f, G_{\tau}\rangle=c_f(\tau, 1)|4\tau|^{\frac{n+1}{2}-k}\mu(k, n). 
\]
\end{enumerate}
\end{proposition}

\begin{proof} 
\textbf{(i)} To show $G_{\tau}\in\mathcal{S}_k(\Gamma, \psi)$ we show $G_{\tau}||_k\gamma = \psi_{\mfrak{c}}(|a_{\gamma}|) G_{\tau}(z)$ for all 
$\gamma\in\Gamma$. We have $h( \alpha\gamma, z) = h(\alpha, \gamma z)h(\gamma, z)$ by the automorphic property,(\ref{h3}), of the function $h$ and, combined with the usual cocyle relation on $j(\alpha\gamma, z)$, we obtain
\[
	j_{\alpha}^k(\gamma z) = j_{\gamma}^k(z)^{-1}j_{\alpha\gamma}^k(z).
\]
Write $\gamma = pw$ with $p\in \Gamma_{\infty}$ and 
$w\in \Gamma_{\infty}\bslsh \Gamma$. Then $\alpha\gamma = \alpha w$ in $\Gamma_{\infty}
\bslsh\Gamma$, and the map $\alpha\mapsto \alpha w$ is both a bijection and well defined on 
$\Gamma_{\infty}\bslsh\Gamma$. With all of this, and noting $a_{\alpha\gamma}\equiv a_{\alpha}a_{\gamma}\pmod{\mfrak{c}}$, we get 
\begin{align*}
	G_{\tau}(\gamma z) &= j_{\gamma}^k(z)\psi_{\mfrak{c}}(|a_{\gamma}|)
\sum_{\alpha w\in \Gamma_{\infty}\bslsh\Gamma} \psi_{\mfrak{c}}^{-1}(|a_{\alpha q}|) 
j_{\alpha w}^k(z)^{-1} e(\tr(\tau(\alpha w)z)) \\
	&= j_{\gamma}^k(z)\psi_{\mfrak{c}}(|a_{\gamma}|) G_{\tau}(z).
\end{align*}
For convergence, the integral case was proved in \cite{Godement} and, by the property in (\ref{h1}) 
of the function $h$, we have
\[ 
	|j(\gamma, z)^{k'}|<|j_M^k(z)|<|j(\gamma, z)^{k''}| 
\]
where $k'=[k]$ and $k''=k'+1$ if $|j(\gamma, z)|>1$ and 
$k'=[k]+1, k''=k'-1$ if $|j(\gamma, z)|<1$. So absolute convergence, 
uniformly on compact subgroups of $\mathbb{H}_n$, follows from the integral case.

\textbf{(ii)} and \textbf{(iii)} can be proven in precisely the same manner as the integral 
case in \cite[p. 332]{Sturm} using this adapted Poincar\'{e} series in place of the one appearing 
there. Note that for (iii) we require boundedness of $|y|^{\frac{k}{2}}|f(z)|$, something which 
is clarified for the half-integral weight case later on in the proof of Corollary \ref{boundsonemod} (i).
\end{proof}

With Proposition \ref{poincare} above the proof of Theorem \ref{holoproj} follows by setting
\[
	K(z, w) := N(\mfrak{b})^{\frac{n(n+1)}{2}}\Vol(\Gamma\bslsh\mathbb{H}_n)\mu(k, n)^{-1}\sum_{0<\tau\in S_{\mfrak{b}}^{\triangledown}}|4\tau|^{k-\frac{n+1}{2}}G_{\tau}(z)e(-\tr(\tau\bar{w}))
\]
and proceeding as in \cite[pp. 332--333]{Sturm}.

In the rest of the section we extend some bounds found in \cite[pp. 335--336]{Sturm} to our 
setting; these bounds shall govern when holomorphic projection is applicable in certain cases.
Let $\kappa\in\frac{1}{2}\mathbb{Z}$ and $\Gamma_0$ be a congruence subgroup that is contained in $\mfrak{M}$ if $\kappa\notin\mathbb{Z}$. Define, for a variable $z\in\mathbb{H}_n$ and $b\in\mathbb{R}$ such that $b>\frac{n+1}{2}$, the following majorant of the non-holomorphic Eisenstein series
\[ 
	H_{\kappa}(z, b; \Gamma_0)=H_{\kappa}(z, b):=|y|^{b - \frac{\kappa}{2}}\sum_{\alpha\in P\cap\Gamma_0\bslsh \Gamma_0}\|c_{\alpha}z+d_{\alpha}\|^{-2b}.
\]

Let $\Omega$ be a fundamental domain for $Sp_n(\mathbb{Z})\bslsh\mathbb{H}_n$ chosen 
so that $z=x+iy\in\Omega$ implies that $y>\eps I_n$ for some $\eps>0$ independent of $z$.

\begin{proposition} \label{alt} 
Let $C_0, a\in\mathbb{R}$ be given with $C_0>0$ and $a\geq 0$. Let 
$\ph:\mathbb{H}_n\to\mathbb{C}$ be such that
\[ 
	|\ph^2(\gamma\cdot z)|\leq C_0|y|^a 
\]
for all $z\in\Omega$ and $\gamma\in Sp_n(\mathbb{Z})$. Then, writing $\lambda_j$ as the eigenvalues of $y$ and taking only positive square roots, we have
\[ 
	|\ph(z)|\leq C_1\prod_{j=1}^n(\lambda_j^{\frac{a}{2}}+\lambda_j^{-\frac{a}{2}}),
\]
for some constant $C_1>0$ dependent only on $\ph$.
\end{proposition}

\begin{proof} 
Let $z\in\mathbb{H}_n$ and choose $\gamma
\in Sp_n(\mathbb{Z})$ such that $\gamma z\in\Omega$. Then 
\begin{align} 
	|\ph^2(z)| = |\ph^2(\gamma^{-1}(\gamma z))|\leq C_0|\Im(\gamma z)|^a 
= C_0|y|^a\|c_{\gamma}z+d_{\gamma}\|^{-2a}.  \label{phisq}
\end{align}
Let $r$ be the rank of $c_{\gamma}$; as in \cite[p. 334]{Sturm} there exist 
$U_1, U_2\in GL_n(\mathbb{Z})$ such that
\[ 
	c=U_1\begin{pmatrix} c_1 & 0 \\ 0 & 0\end{pmatrix} U_2^T, \hspace{10pt} 
d=U_1\begin{pmatrix} d_1 & 0 \\ 0 & I_{n-r}\end{pmatrix} U_2^{-1},
\]
where $c_1, d_1\in M_r(\mathbb{Z})$ are such that $|c_1|\neq 0$ and $c_1d_1^T$ is symmetric. 
Put $U_2=\begin{pmatrix} Q &  Q'\end{pmatrix}$, where $Q\in M_{n\times r}(\mathbb{Z}), Q'\in M_{n\times (n-r)}(\mathbb{Z})$. 
Then we have $\|c_{\gamma}z+d_{\gamma}\| = \| Q^TzQ + c_1^{-1}d_1\|\geq \|QyQ^T\|$ so that, from (\ref{phisq}) above,
\[ 
	|\ph^2(z)|\leq C_0|y|^a|y_0|^{-2a},
\]
where $y_0=Qy Q^T$.
Sturm shows, in \cite[p. 334]{Sturm}, that there exist $1\leq j_1<j_2<\cdots <j_r\leq n$ such that  $|y_0|\geq \alpha\prod_{\nu=1}^r \lambda_{j_{\nu}}$.
Now we have $\lambda_j>0$ for all $j$ and so 
$\prod_{\nu=1}^r\lambda_{j_{\nu}}^{-a}\leq \prod_{j=1}^n (1+\lambda_j^{-a})$
-- the left-hand side is just one term in the expansion on the right-hand side, 
all terms of which are $\geq 0$. So, with $C_1=\sqrt{C_0}\alpha^{-a}$, we indeed get
\[ 
	|\ph(z)|\leq C_1\prod_{j=1}^n\lambda_j^{\frac{a}{2}}(1+\lambda_j^{-a})=C_1\prod_{j=1}^n(\lambda_j^{\frac{a}{2}}+\lambda_j^{-\frac{a}{2}}).
\]
\end{proof}

\begin{corollary}\label{boundsonemod}
Let $f\in\mathcal{S}_k(\Gamma, \psi)$; $g\in\mathcal{M}_{\ell}(\Gamma, \psi)$; $\ell, \kappa\in\frac{1}{2}\mathbb{Z}$; and $b>\frac{n+1}{2}$. Then there exists a constant $0<C\in\mathbb{R}$ such that
\begin{enumerate}[(i)]
	\item \[ |f(z)|\leq C|y|^{-\frac{k}{2}}; \]
	\item \[ |g(z)|\leq C\prod_{j=1}^n (1-\lambda_j^{-\ell}); \]
	\item \[ |H_{\kappa}(z, b)|\leq C\prod_{j=1}^n(\lambda_j^{b-\frac{\kappa}{2}}+\lambda_j^{-b-\frac{\kappa}{2}}). \]
\end{enumerate}
\end{corollary}

\begin{proof}
\textbf{(i)} Consider $f^2$ -- a cusp form of integral weight $2k$ and level $\Gamma$. 
Apply the above Proposition \ref{alt} to the 
function $\ph(z):=|y|^{\frac{k}{2}}f(z)$ with $a=0$.

\indent \textbf{(ii)} Consider $g^2$ -- a modular form of integral weight $2\ell$ and level $\Gamma$. Then $\ph(z) := |y|^{\frac{\ell}{2}}g(z)$ satisfies the conditions of Proposition \ref{alt} with $a=\frac{2\ell}{2}=\ell$.

\indent \textbf{(iii)} $H^2$ is a constant multiple of $H_{2\kappa}(z, 2b)$ which is of integral weight $2\kappa$. Sturm shows,  in \cite[p. 335]{Sturm}, that $|\Im(\gamma z)|^{\kappa}|H_{2\kappa}(\gamma z, 2b)|\leq C_0|y|^{2b}$. Hence $\ph(z) := |y|^{\frac{\kappa}{2}}H_{\kappa}(z, b)$ satisfies the conditions of Proposition \ref{alt} with $a = 2b$.
\end{proof}

\begin{corollary} \label{boundstwo} 
Let $k$ be a half-integral weight, $\ell\in\frac{1}{2}\mathbb{Z}$, $g\in\mathcal{M}_{\ell}(\Gamma, \psi)$, $b>\frac{n+1}{2}$, and put $F^*(z):=g(z)H_{k-\ell}(z, b)$.
Then $F^*$ is of bounded growth provided we have
\[ 
	\frac{n+1}{2}<b<
\begin{cases} 
	\frac{k}{2}-n&\text{if $g\in\mathcal{S}_{\ell}(\Gamma, \psi)$}, \\ 
	\frac{k-\ell}{2}-n &\text{otherwise}. 
\end{cases} 
\]
\end{corollary}

The proof of the above corollary is precisely as it appears in \cite[pp. 335--336]{Sturm}, 
since we have the same setup with Corollary \ref{boundsonemod}.

\section{Integral expressions for the standard metaplectic $L$-function}
\label{integralsection}

The main object of study --  the standard, twisted $L$-function $L_{\psi}(s, f, \chi)$ of an eigenform $f$ -- is introduced here and an integral expression, from \cite{Shimuraexp}, is taken and modified for our purposes. Throughout let $\delta := n\pmod{2}\in\{0, 1\}$.

Though the integral expression we obtain can be stated 
for any half-integral weight $k$, for ease of notation we take 
$k\geq n+1$ -- we shall be making this assumption later on anyway. 
For a prime $p$, the association of an $n$-tuple $(\lambda_{p, 1}, \dots, \lambda_{p, n})\in\mathbb{C}^n$ 
to a non-zero Hecke eigenform $f\in\mathcal{S}_k(\Gamma, \psi)$ is well known; this process is briefly outlined later in Sect. \ref{eisensteinsubsect1}. Then, for a Hecke character $\chi$ of conductor $\mfrak{f}$, we define our $L$-function by
\begin{align*}
	L_p(t)&:=
\begin{cases} 
	\displaystyle\prod_{i=1}^n (1-p^{n}\lambda_{p, i}t) &\text{if $p\mid \mfrak{c}$}, \\  
	\displaystyle\prod_{i=1}^n (1-p^{n}\lambda_{p, i}t)(1-p^{n}\lambda_{p, i}^{-1}t) 
&\text{if $p\nmid \mfrak{c}$};
\end{cases} \\
	\displaystyle L_{\psi}(s, f, \chi)&:=\prod_p L_p\left((\psi^{\mfrak{c}}
\chi^*)(p)p^{-s}\right)^{-1},
\end{align*}
in which $\chi^*(p) = \chi^*(p\mathbb{Z})$ is the ideal Hecke character associated to $\chi$ and
\[
	\psi^{\mfrak{c}}(x) := \left(\frac{\psi}{\psi_{\mfrak{c}}}\right)(x).
\]

Fix $\tau\in S_+$ such that $c_f(\tau, 1)\neq 0$ and let $\rho_{\tau}$ be the quadratic character associated to the extension $\mathbb{Q}(i^{[n/2]}\sqrt{|2\tau|})$; define $\mfrak{t}$ as an ideal in $\mathbb{Z}$ such that $h^T(2\tau)h\in 4\mfrak{t}^{-1}$ for all $h\in\mathbb{Z}^n$. Choose $\mu\in\{0, 1\}$ such that $(\psi\chi)_{\infty}(x) = \sgn(x_{\infty})^{[k]+\mu}$. The integral expression (4.1) in \cite[p. 342]{Shimuraexp} is stated there in immense generality and a lot of this simplifies in our setting, for example, in the notation of \cite{Shimuraexp}, we can just take $p = I_n$ and $D_F = 1$. The key ingredients of the integral are three modular forms: the eigenform $f$, a theta series $\theta$, and a normalised Eisenstein series $\mathcal{E}(z, s)$. If $\mu\in\{0, 1\}$ and $\chi$ is a Hecke character of conductor $\mfrak{f}$ satisfying $\chi_{\infty}(x)^n = \sgn(x_{\infty})^{n\mu}$, then the definition of the theta series $\theta$, taken from \cite[(2.1)]{Shimuraexp}, is 
\[
	\theta_{\chi}(z) = \theta_{\chi}^{(\mu)}(z; \tau) := \sum_{x\in M_n(\mathbb{Z})}(\chi_{\infty}\chi^*)^{-1}(|x|)|x|^{\mu}e_{\infty}(\tr(x^T\tau x z)).
\]
This has weight $\frac{n}{2}+\mu$, level determined by Proposition 2.1 of \cite{Shimuraexp}, character $\rho_{\tau}\chi^{-1}$, and coefficients in $\mathbb{Q}(\chi)$.

We define the Eisenstein series of weight $\kappa\in\frac{1}{2}\mathbb{Z}$ in a little more generality. Let $\Gamma_0 = \Gamma[\mfrak{x}^{-1}, \mfrak{xy}]$ be a congruence subgroup (contained in $\mfrak{M}$ if $\kappa\notin \mathbb{Z}$) and let $\ph$ be a Hecke character satisfying (\ref{char1}) with $\mfrak{y}$ in place of $\mfrak{c}$ and 
\begin{align}
	\ph_{\infty}(x) = \sgn(x_{\infty})^{[\kappa]}.\label{char3}
\end{align}
Note the above condition is more stringent than the usual parity condition of (\ref{char2}) when $n$ is even. Then the non-holomorphic Eisenstein series is defined as
\[
	E_{\kappa}(z, s;  \ph, \Gamma') := |y|^{s-\frac{\kappa}{2}} \sum_{\alpha\in P\cap\Gamma_0\bslsh\Gamma_0}\ph_{\mfrak{y}}(|a_{\gamma}|)j_{\gamma}^{\kappa}(z)^{-1}|\mu(\gamma, z)|^{\kappa-2s},
\]
for variables $z\in\mathbb{H}_n$, $s\in\mathbb{C}$. This sum is convergent for $\Re(s)>\frac{n+1}{2}$ \cite[p. 133]{Shimurabook}, and can be continued analytically to all of $s\in\mathbb{C}$ by a functional equation in $s\mapsto \frac{n+1}{2}-s$. This series belongs to $C_{\kappa}(\Gamma_0, \ph^{-1})$, and we normalise it using a product of Dirichlet $L$-functions as follows. Let $\mfrak{a}$ be any integral ideal and define
\begin{align}
\begin{split}
	L_{\mfrak{a}}(s, \varphi)&:=\prod_{p\nmid\mfrak{a}}(1-\varphi^*(p)p^{-s})^{-1}; \\
	\Lambda_{\mfrak{a}}^{n, \kappa}(s, \varphi)&:=\begin{cases}
		\displaystyle L_{\mfrak{a}}(2s, \varphi)
\prod_{i=1}^{[\frac{n}{2}]} L_{\mfrak{a}}(4s-2i, \varphi^2) &\text{if $\kappa\in\mathbb{Z}$}, \\
		\displaystyle\prod_{i=1}^{[\frac{n+1}{2}]} L_{\mfrak{a}}(4s-2i+1, \varphi^2) &\text{if $\kappa\notin\mathbb{Z}$}.
	\end{cases} 
\end{split}
\label{DirL}
\end{align}
The normalised Eisenstein series is given by
\[
	\mathcal{E}_{\kappa}(z, s) = \mathcal{E}_{\kappa}(z, s; \ph, \Gamma_0) := \overline{\Lambda_{\mfrak{y}}^{n, \kappa}(s, \ph)}E(z, \bar{s}; \kappa, \ph, \Gamma_0).
\]
Then the integral expression of \cite[(4.1)]{Shimuraexp} is:
\begin{align}\begin{split}
	L_{\psi}(s, f, \chi)&=\left[\Gamma_n
\left(\tfrac{s-n-1+k+\mu}{2}\right)2c_f(\tau, 1)\right]^{-1}N(\mfrak{b})^{\frac{n(n+1)}{2}}|4\pi\tau|^{\frac{s-n-1+k+\mu}{2}}\\
&\times \left(\tfrac{\Lambda_{\mfrak{c}}}{\Lambda_{\mfrak{y}}}\right)
\left(\tfrac{2s-n}{4}\right)\prod_{p\in\mathbf{b}}g_p\left((\psi^{\mfrak{c}}\chi^*)(p)p^{-s}\right)
	\langle f, \theta
\mathcal{E}\left(\cdot, \tfrac{2s-n}{4}\right)\rangle V, \label{intexp}
\end{split}
\end{align}
where 
\[
	\Lambda_{\mfrak{a}}(s) := \Lambda_{\mfrak{a}}^{n, k-n/2-\mu}(s, \eta);
\]
$\mfrak{y} = \mfrak{c}\cap(4\mfrak{tf}^2)$; $\mathbf{b}$ is a finite set of primes and $g_p\in\mathbb{Z}[t]$ are polynomials, defined for each $p\in\mathbf{b}$, such that $g_p(0) = 1$ (these arise as integral factors of a certain Dirichlet series $\alpha$, see Theorem 5.2 of \cite{Shimurahalf}); $\eta = \psi\chi\rho_{\tau}$; 
\[
	\mathcal{E}(z, s) := \mathcal{E}(z, s; k-\tfrac{n}{2}-\mu, \bar{\eta}, \Gamma[\mfrak{b}^{-1}, \mfrak{by}]);
\]
and $V := \Vol(\Gamma[\mfrak{b}^{-1}, \mfrak{by}]\bslsh\mathbb{H}_n)$.

Notice, by the definitions of (\ref{DirL}) above, that $\left(\frac{\Lambda_{\mfrak{c}}}{\Lambda_{\mfrak{y}}}\right)(\frac{2s-n}{4})$ is just a finite product of Euler factors twisted by $\eta$ and denote it by $\Lambda_{\mfrak{c}, \mfrak{y}}(s, \eta)$. We have $\Lambda_{\mfrak{c}, \mfrak{y}}(s, \eta)^{\sigma} = \Lambda_{\mfrak{c}, \mfrak{y}}(s, \eta^{\sigma})$ for any $\sigma\in\Aut(\mathbb{C})$.

We need some knowledge of algebraicity of our Eisenstein series; this is given by Theorem 3.2 in
\cite{Bouganis} and is restated next. For any $\ell\in\frac{1}{2}\mathbb{Z}$ define the set
\[
	\Omega_0:=\left\{ s\in\tfrac{1}{2}\mathbb{Z}\bigg| \left|s-\tfrac{n+1}{4}\right|+\tfrac{n+1}{4}-\tfrac{k-\ell}{2}\in \mathbb{Z}, 
\tfrac{n+1-k+\ell}{2}\leq s\leq \tfrac{k-\ell}{2}\right\}
\]
and, for any Hecke character $\ph$ of conductor $\mfrak{f}$, its Gauss sum to be
\[
	G(\ph) := \sum_{a=1}^{N(\mfrak{f})} \ph_{\mfrak{f}}^{-1}(a)e^{\frac{2\pi ia}{N(\mfrak{f})}}.
\]
There are exceptional cases where the Eisenstein series $\mathcal{E}_{\kappa}(z, \frac{2m-n}{4}; \ph, \Gamma_0)$ has different behaviour. The relevant ones are:
\begin{align}
	m &= n+1\ \text{and $\varphi^2 = 1$}; \label{X} \tag{\textbf{X}}\\
	n&=1, m = \tfrac{3}{2},\ \text{and $\varphi = 1$}; \label{R1}\tag{\textbf{R1}} \\
	n&>1, m = n+\tfrac{3}{2},\ \text{and $\varphi^2 = 1$}. \label{R2} \tag{\textbf{R2}}
\end{align}
Case (\ref{X}) affects neither of the main results, Theorems \ref{main} and \ref{newmain}, since the set of special values are strict half-integers. Neither Case (\ref{R1}) nor Case (\ref{R2}) affect the first main result, Theorem \ref{main}, since the set of special values excludes them, but they will have an effect on the second main result, Theorem \ref{newmain}.

\begin{theorem}[(Bouganis, \cite{Bouganis}, Th. 3.2)]\label{eisenalg} Fix a half-integral weight $k$, let $\ell\in\frac{1}{2}\mathbb{Z}$ satisfy $k-\ell>\frac{n+1}{2}$. Put $\Gamma_0 = \Gamma[\mfrak{x}^{-1}, \mfrak{xy}]$ (contained in $\mfrak{M}$ if $k-\ell\notin \mathbb{Z})$. Let $\ph$ be a Hecke character satisfying the usual property of (\ref{char1}) with $\mfrak{y}$ in place of $\mfrak{c}$, as well as (\ref{char3}) with $\kappa = k-\ell$. Exclude case (\ref{X}). For any $m$ with $\frac{2m-n}{4}\in \Omega_0$ we have 
\[
	\mathcal{E}_{k-\ell}(z, \tfrac{2m-n}{4}; \ph, \Gamma_0) = |\pi y|^{-r}\sum_{0\leq\tau\in S_{\mfrak{x}}^{\triangledown}} P(\tau, \varphi, y)e(\tr(\tau z)),
\]
where $P(\tau, \varphi, y)\in\mathbb{Q}_{ab}[\pi y_{ij}\mid 1\leq i\leq j\leq n]$, and
\[
	r := \begin{cases}  \frac{k-\ell}{2} - \frac{2m-n}{4} + 1 &\text{in cases (\ref{R1}) and (\ref{R2})}, \\
	\frac{k-\ell}{2} - |\frac{2m-2n-1}{4}| - \frac{n+1}{4} &\text{otherwise}.
	\end{cases}
\]
Set 
\[
	\beta_m := \tfrac{n}{2}(k-\ell+m-n)+\tfrac{\delta}{4},
\]
and define a period $\omega_{\ell}(m, \varphi) = \omega_{\ell}(\varphi) = \omega(\varphi)$ by
\[
	\omega(\varphi):=\begin{cases} i^{n|\frac{2k-2\ell-2m+n}{4}|+mn-\frac{3n^2-1}{4}} G(\varphi)G(\varphi^{n-1}) & \text{if $k-\ell\in\mathbb{Z}$ and  $m>n$}; \\
	i^{n|\frac{2k-2\ell-3n-2+2m}{4}|-\frac{n}{2}(3n+2-2m)}G(\varphi)^n &\text{if $k-\ell\in\mathbb{Z}$ and $m\leq n$}; \\
	i^{n|\frac{2k-2\ell-2m+n}{4}|-n(k-\ell)}G(\varphi\zeta)^n\zeta_8 &\text{if $k-\ell\notin\mathbb{Z}$, $n\in 2\mathbb{Z}$, $m> n$}; \\
	i^{n|\frac{2k-2\ell-2m+n}{4}|-n(k-\ell)+\nu}(2i)^{\frac{3n}{2}-m}G(\varphi\zeta)^nG(\ph)\zeta_8 &\text{if $k-\ell\notin\mathbb{Z}$, $n\notin 2\mathbb{Z}$, $m>n$}; \\
	i^{n|\frac{2k-2\ell-3n-2+2m}{4}|-\frac{n}{2}(3n+2-2m)}G(\varphi\zeta)^n\zeta_8 &\text{if $k-\ell\notin\mathbb{Z}$ and $m\leq n$};
\end{cases} 
\]
where $\zeta_8$ is a fixed eighth root of unity, $\zeta$ is the character induced by $h_{\gamma}(z)^2 = \zeta(\gamma)j(\gamma, z)$, and $\nu = 1$ if $n\equiv 1\pmod{4}$ but $\nu = 0$ otherwise.
Then we have
\[
	\left[\frac{P(\tau, \varphi, y)}{\pi^{\beta_m}\omega(\varphi)}\right]^{\sigma} = \frac{P(\tau, \varphi^{\sigma}, y)}{\pi^{\beta_m}\omega(\varphi^{\sigma})}
\]
for any $\sigma\in\text{Gal}(\mathbb{Q}_{ab}/\mathbb{Q})$ (acting on $P(\tau, \ph, y)$ coefficients wise).
\end{theorem}

Let $\ell\in\frac{1}{2}\mathbb{Z}$ and fix $g\in \mathcal{M}_{\ell}(\Gamma', \psi')$,
where $\Gamma' = \Gamma[(\mfrak{b}')^{-1}, \mfrak{b}'\mfrak{c}']\leq\mfrak{M}$ if $\ell\notin\mathbb{Z}$ and $\psi'$ is a Hecke character satisfying the properties in (\ref{char1}) and (\ref{char2}) with $\mfrak{c} = \mfrak{c}'$ and $\kappa = \ell$. Let $k$ be a half-integral weight and set
\[ 
	\Omega_g':=
\begin{cases} 
	\left\{m\in\mathbb{R}\mid \tfrac{n-2m+2k-2\ell}{4}\in \mathbb{Z}, \tfrac{3n}{2}+1<m<k-\ell-\tfrac{3n}{2}\right\} 
&\text{if $g\notin\mathcal{S}_{\ell}$}; \\ 
	\left\{m\in\mathbb{R}\mid \tfrac{n-2m+2k-2\ell}{4}\in \mathbb{Z}, \tfrac{3n}{2}+1<m<k-\tfrac{3n}{2}, 
m\leq k-\ell+\tfrac{n}{2}\right\} &\text{if $g\in\mathcal{S}_{\ell}$}.
\end{cases}
\]
If $f\in\mathcal{M}_k(\Gamma, \psi)$ then we assume that $\psi/\psi'$ satisfies
\[
	(\psi/\psi')_{\infty}(x) = \sgn(x_{\infty})^{[k-\ell]},
\]
which is always satisfied if $n$ is odd by the properties of $\psi$ and $\psi'$ (\ref{char2}), and ensures that we can define the Eisenstein series
\[
	\mathcal{E}_{k-\ell}^{\bar{\psi}\psi'}(z, s) := \mathcal{E}_{k-\ell}(z, s; \bar{\psi}\psi', \Gamma\cap\Gamma').
\]
Recall that $\delta = n\pmod{2}\in\{0, 1\}$.

\begin{proposition}\label{newintegralexp} 
Let $k>2n$, $\ell$, and $g$ be as above. Set 
\[
	m_0:=\tfrac{2k+2\ell+2m-n}{4}
-\tfrac{n+1}{2}.
\]
For every $m\in\Omega'_g$ there exists $K(m, g)\in \mathcal{S}_k(\Gamma, \psi)$, 
whose Fourier coefficients lie in $\mathbb{Q}_{ab}(g)$, such that
\begin{align*}
	\frac{(4\pi)^{nm_0}}{\pi^{\beta_m}\omega_{\ell}(m, \bar{\psi}\psi')}\Gamma_n(m_0)^{-1}&\left\langle f, g\mathcal{E}_{k-\ell}^{\bar{\psi}\psi'}\left(\cdot, \tfrac{2m-n}{4}\right)\right\rangle=\pi^{n(k-r)-\frac{3n^2+2n+\delta}{4}}\langle f, K(m, g)\rangle
\end{align*}
for all $f\in \mathcal{S}_k(\Gamma, \psi)$. We also have $K(m, g)^{\sigma} = 
K(m, g^{\sigma})$ for all $\sigma\in\Aut(\mathbb{C})$.
\end{proposition}

Before proving this, we require the following lemma, whose statement and proof is adapted from \cite[p. 343]{Sturm}.

\begin{lemma}\label{rationalintegral} 
Let $0<\tau\in S_{\mfrak{b}}^{\triangledown}$ and $P(y)\in\mathbb{Q}[y_{ij}\mid i\leq j]$. If $\nu$ is such that $\nu-\frac{1}{2}\in\mathbb{Z}$ and $\nu>n$, then we have
\[ 	
	|\tau|^{\frac{n}{2}}\Gamma_n\left(\nu-\tfrac{n+1}{2}\right)^{-1}\int_Y 
P(y)e^{-\tr(\tau y)}|y|^{\nu-\frac{n+1}{2}}d^{\times}y\in\mathbb{Q}. 
\]
\end{lemma}

\begin{proof} 
We can assume that $P(y) = \prod_{i\leq j}y_{ij}^{a_{ij}}$, where $0\leq a_{ij}\in\mathbb{Z}$. If $U = (U_{ij})\in Y$ then by definition
\[
	\int_Y |y|^{\nu-\frac{n+1}{2}}e^{-\tr(Uy)}d^{\times}y = \Gamma_n(\nu-\tfrac{n+1}{2})|U|^{\frac{n+1}{2}-\nu}
\]
and apply $\prod_{i\leq j}\left(\frac{\partial}{\partial U_{ij}}\right)^{a_{ij}}|_{U_{ij} = \tau_{ij}}$ to both sides. This gives
\[
	\int_Y P(y)e^{-\tr(\tau y)}|y|^{\nu-\frac{n+1}{2}}d^{\times}y\in \Gamma_n(\nu-\tfrac{n+1}{2})|\tau|^{\frac{n+1}{2}-\nu}\mathbb{Q},
\]
which, since $\frac{1}{2}-\nu\in\mathbb{Z}$, gives the lemma.
\end{proof}

\begin{proof}[Proof of Proposition~{\rm\ref{newintegralexp}}]
The function 
$\mathcal{F}(z)=g(z)\mathcal{E}_{k-\ell}^{\bar{\psi}\psi'}(z, \frac{2m-n}{4})\in C_k^{\infty}(\Gamma, \psi)$
has bounded growth if and only if
$F(z) := g(z)E_{k-\ell}(z, \frac{2m-n}{4};\bar{\psi}\psi', \Gamma\cap\Gamma')$ does. Since the Eisenstein series in $F$ is bounded above by $H_{k-\ell}(z, \frac{2m-n}{4}; \Gamma\cap\Gamma')$ we invoke Corollary \ref{boundstwo}.
That $\frac{2m-n}{4}$ for $m\in\Omega_g'$ satisfies the inequalities 
in Corollary \ref{boundstwo} is easy to check and 
so we indeed get bounded growth of $F$.

Therefore apply Theorem \ref{holoproj}; set
$\widetilde{K}(m, g)=\frac{(4\pi)^{nm_0}}{\pi^{\beta}\omega(\bar{\psi}\psi')}\Gamma_n(m_0)^{-1}\mathbf{Pr}(\mathcal{F})\in\mathcal{S}_k(\Gamma, \psi)$, and we get
\[ 
	\frac{(4\pi)^{nm_0}}{\pi^{\beta_m}\omega_{\ell}(m, \bar{\psi}\psi')}\Gamma_n(m_0)^{-1}\left\langle f, g
\mathcal{E}_{k-\ell}^{\bar{\psi}\psi'}(\cdot, \tfrac{2m-n}{4})\right\rangle =\langle f, \widetilde{K}(m, g)\rangle. 
\]
The Fourier coefficients of $\mathcal{F}$ are given, by Theorem \ref{eisenalg}, as
\[ 
	c_{\mathcal{F}}(\tau, y) = \sum_{\tau_1+\tau_2 = \tau} c_g(\tau_1, 1)
P(\tau_2, \bar{\psi}\psi',  y)|\pi y|^{-r}e^{-2\pi \tr(\tau y)},
\]
where in this case $r = \frac{n-2m+2k-2\ell}{4}$. So, by Theorem \ref{holoproj}, we obtain the Fourier expansion
\begin{align*}
	\widetilde{K}(m, g) &= \sum_{0<\tau\in S_{\mfrak{b}}^{\triangledown}}\left(\sum_{\tau_1+\tau_2=\tau}a(\tau_1, \tau_2)\right)
e(\tr(\tau z)), \\
	a(\tau_1, \tau_2) &= \mu(k, n)^{-1}c_g(\tau_1, 1)|4\tau|^{k-\frac{n+1}{2}}\frac{(4\pi)^{nm_0}}{\pi^{\beta_m + nr}\omega_{\ell}(m, \bar{\psi}\psi')}
\Gamma_n(m_0)^{-1}\\
	&\hspace{10pt}\times\int_Y P(\tau_2, \bar{\psi}\psi', y)e^{-4\pi\tr(\tau y)}|y|^{k-r-\frac{n+1}{2}}
d^{\times}y.
\end{align*}
The polynomial $P(\tau_2, \bar{\psi}\psi', y)$ has the form
\[
	P(\tau_2, \bar{\psi}\psi', y) = \sum_{1\leq i\leq j\leq n}\sum_{\alpha_{ij}}b_{ij}(\tau_2,\bar{\psi}\psi')\prod_{i, j}(\pi y_{ij})^{\alpha_{ij}},
\] 
for coefficients $b_{ij} = b_{ij}(\tau_2, \bar{\psi}\psi')\in\mathbb{Q}_{ab}$ and $\alpha_{ij}$ summing over some finite set of non-negative integers.
We have that $\mu(k, n)^{-1}\in\pi^{nk-\frac{3n^2+2n+\delta}{4}}\mathbb{Q}$ and that $|4\tau|^{k-\frac{n+1}{2}}
\in |\tau|^{\frac{n}{2}}\mathbb{Q}$. Note also that $m_0 = k-r-\frac{n+1}{2}$, that $k-r>n$, and that the substitution $y\mapsto (4\pi)^{-1}y$ gives us

\begin{align*}
	a(\tau_1, \tau_2) &\in \frac{c_g(\tau_1, 1)\pi^{nk-\frac{3n^2+2n+\delta}{4}}}{\pi^{\beta_m + nr}\omega_{\ell}(m, \bar{\psi}\psi')}\sum_{i, j, \alpha_{ij}}b_{ij}\left[|\tau|^{\frac{n}{2}}\Gamma_n(m_0)^{-1}\int_Y\left(\prod_{ij}y_{ij}^{\alpha_{ij}}\right)e^{-\tr(\tau y)}|y|^{m_0}\right]\mathbb{Q} \\
	&\in c_g(\tau_1, 1)\pi^{n(k-r)-\frac{3n^2+2n+\delta}{4}}\sum_{i, j}\frac{b_{ij}(\tau_2, \bar{\psi}\psi')}{\pi^{\beta_m}\omega_{\ell}(m, \bar{\psi}\psi')}\mathbb{Q},
\end{align*}
where we used Lemma \ref{rationalintegral}. We then obtain the result by putting $K(m, g) = \pi^{\frac{3n^2+2n+\delta}{4}-n(k-r)}\tilde{K}(m, g)$ and directly applying Theorem \ref{eisenalg}.
\end{proof}

If $m\in\Omega_{\theta_{\chi}}'$ then we can make some
simplifications to our integral expression of (\ref{intexp}). By putting $\ell = \frac{n}{2}+\mu$ into the notation of Proposition \ref{newintegralexp} we get $m_0 = \frac{m-n-1+k+\mu}{2}\in\frac{n}{2}+\mathbb{Z}$. Note that $|\tau|^{\frac{1}{2}}\in 2^{\frac{1}{2}}|2\tau|^{-\frac{1}{2}}\mathbb{Q}$ and $V\in \pi^{\frac{n(n+1)}{2}}\mathbb{Q}$. Then from the integral expression of (\ref{intexp}) we obtain
\begin{align*}
	L_{\psi}(m, f, \chi)\in &(2|2\tau|)^{-\frac{\delta}{2}}c_f(\tau, 1)^{-1}\Lambda_{\mfrak{c}, \mfrak{y}}(m, \eta)\prod_{p\in\mathbf{b}}g_p\left((\psi^{\mfrak{c}}\chi^*)(p)p^{-m}\right) \\
	&\times \pi^{\frac{n(n+1)}{2}}(4\pi)^{nm_0}\Gamma_n(m_0)^{-1}\left\langle f, \theta_{\chi}\mathcal{E}(\cdot, \tfrac{2m-n}{4})\right\rangle\mathbb{Q}.
\end{align*}
Notice that, if $\ell\in\frac{n}{2}+\mathbb{Z}$ we have $k-\ell\in \mathbb{Z}$ if and only if $n$ is odd, so relabel $\omega_{\delta}(\ph) := \omega_{\ell}(\ph)$ in this case. Multiplying both sides by $\pi^{-\beta_m}\omega_{\delta}(m, \bar{\eta})^{-1}$, and then applying Proposition \ref{newintegralexp} gives
\begin{align}
	\frac{(2|2\tau|)^{\frac{\delta}{2}}c_f(\tau, 1)L_{\psi}(m, f, \chi)}{\pi^{\beta_m+n(k-r)-\frac{3n^2+2n+\delta}{4}}\omega_{\delta}(m, \bar{\eta})} \in \Lambda_{\mfrak{c}, \mfrak{y}}(m, \eta)\prod_{p\in\mathbf{b}}g_p\left((\psi^{\mfrak{c}}\chi^*)(p)p^{-m}\right)\left\langle f, K(m, \theta_{\chi})\right\rangle\mathbb{Q}. \label{projintexp}
\end{align}

\section{Algebraicity of the inner product}
\label{algebraicitysection}

The algebraicity of the right-hand side of the integral expression (\ref{projintexp}), and subsequently of the special values, is immediate except for the inner product $\langle f, K(m, \theta_{\chi})\rangle$.
For a system of eigenvalues $\Lambda:\mathcal{R}_0\to\mathbb{C}$, where $\mathcal{R}_0$ is the space of Hecke operators (defined
in the next section or in \cite[pp. 39--41]{Shimurahalf}), let $\mathcal{S}_k(\Gamma, \psi, \Lambda)$ denote the eigenforms in $\mathcal{S}_k(\Gamma, \psi)$ whose eigenvalues are given by $\Lambda$. By \cite[Lemma 23.14]{Shimurabook} we have $f^{\sigma}\in\mathcal{S}_k(\Gamma, \psi^{\sigma}, \Lambda_{\sigma})$, where
\[
	\Lambda_{\sigma}(A(n)) = \Lambda(A(n))^{\sigma}\frac{(\sqrt{n})^{\sigma}}{\sqrt{n}},
\]
$\sigma\in\Aut(\mathbb{C})$, and $f\in\mathcal{S}_k(\Gamma, \psi, \Lambda)$. Let $\rho\in\Aut(\mathbb{C})$ denote the complex conjugation automorphism, and define $\eps\in\{0, 1\}$ by $\psi_{\infty}(x) = \sgn(x_{\infty})^{[k] + \eps}$. Note that $\eps = 0$ if $n$ is odd by (\ref{char1}).

\begin{theorem}\label{innerprod}
Assume that $k>\frac{7n}{2}+3 +\eps$. If $f\in\mathcal{S}_k(\Gamma[\mfrak{b}^{-1}, \mfrak{bc}], \psi, \Lambda)$ is a Hecke eigenform for ideals $(\mfrak{b}^{-1}, \mfrak{bc})\subseteq 2\mathbb{Z}\times 2\mathbb{Z}$, a Hecke character $\psi$, and a system of eigenvalues $\Lambda$, then there exists a non-zero constant $\mu'(\Lambda, k, \psi)$ -- dependent only on $\Lambda, k$, and $\psi$ -- such that
\[ 
	\left(\frac{\langle f, g\rangle}{\mu'(\Lambda, k, \psi)}\right)^{\sigma}
=\frac{\langle f^{\sigma}, g^{\rho\sigma\rho}\rangle}{\mu'(\Lambda_{\sigma}, k, \psi^{\sigma})},
\] 
for any $g\in\mathcal{S}_k(\Gamma[(\mfrak{b}')^{-1}, \mfrak{b}'\mfrak{c}'], \psi)$, ideals $((\mfrak{b}')^{-1}, \mfrak{b}'\mfrak{c}')\subseteq\mfrak{b}^{-1}\times\mfrak{bc}$, and $\sigma\in\Aut(\mathbb{C}/\mathbb{Q})$.
\end{theorem}

\begin{proof} 
Recall that $\omega_{\delta} = \omega_{\frac{n}{2}+\eps}$ is the period of algebraicity for the Eisenstein series in Theorem \ref{eisenalg}. The constant is defined by
\[
\mu'(\Lambda, k, \psi):=
		2^{\frac{\delta}{2}}\pi^{-a}i^{-\frac{n^2}{2}}\omega_{\ell}(m_{\eps}, \overline{\psi})^{-1}L_{\psi}(m_{\eps}, f), 
\]
where $m_{\eps} := k-2n-2-\eps$, $a=\beta_{m_{\eps}}+n(k-r)-\frac{n^2+\delta}{4}$, $\ell = \frac{n}{2}+\eps$, and $L_{\psi}(s, f) : = L_{\psi}(s, f, 1)$. This constant is non-zero as the Euler 
product of $L_{\psi}(s, f, \varphi)$ is absolutely convergent for values 
$s$ with $\Re(s)>\frac{3n}{2}+1$ -- \cite[p. 332]{Shimuraexp} -- 
of which $s=m_{\eps}$ is such a value by our choice of $k$.
If $\mfrak{m}$ is an integral ideal let $1_{\mfrak{m}}$ denote the ideal Hecke character that is trivial modulo $\mfrak{m}$, and let $\theta_{\mfrak{m}}$ denote the theta series that we have been using, from \cite[(2.1)]{Shimuraexp}, but with $1_{\mfrak{m}}^*$ in place of $\eta^*$ and $\eps$ in place of $\mu$. Since $\eps = 0$ unless $n$ is even, the parity condition needed to define $\theta_{\mfrak{m}}$ is satisfied by $1_{\mfrak{m}}$.

Recall $\mfrak{t}$ as an integral ideal such that $h^T(2\tau)h\in 4\mfrak{t}^{-1}$ for all $h\in\mathbb{Z}^n$ and let $\mfrak{m} = \mfrak{tc}$, then Proposition 2.1 of
 \cite{Shimuraexp} gives the level $\Gamma[2, 2\mfrak{t}^3\mfrak{c}^2]$ of $\theta_{\mfrak{m}}$. Hence, we have $\theta_{\mfrak{m}}\in\mathcal{M}_{\frac{n}{2}+\eps}(\Gamma[\mfrak{b}^{-1}, \mfrak{bc'}], \rho_{\tau})$ and $f\in\mathcal{S}_k(\Gamma[\mfrak{b}^{-1}, \mfrak{bc}'], \psi)$, where $\mfrak{c}' := 2\mfrak{b}^{-1}\mfrak{t}^3\mfrak{c}^2$. Since
\[ 
	\Omega'_{\theta_{\mfrak{m}}}=\{m\in\tfrac{1}{2}\mathbb{Z}\mid\tfrac{k-m-\eps}{2}\in\mathbb{Z}, 
\tfrac{3n}{2}+1<m<k-2n-\eps\},
\]
we can take $m_{\eps} = k-2n-2-\eps\in\Omega'_{\theta_{\mfrak{m}}}$ and apply Proposition \ref{newintegralexp} with $g = \theta_{\mfrak{m}}$ to obtain the integral expression (\ref{projintexp}), with $\chi = 1_{\mfrak{m}}$, for $L_{\psi}(m_{\eps}, f, 1_{\mfrak{m}})$. 

The character appearing in the Eisenstein series of the integral expression is $\eta = \psi\rho_{\tau}$. We want to compare the period $\omega_{\delta}(m_{\eps}, \bar{\psi})$ appearing in the definition of $\mu'(\Lambda, k, \psi)$ with $\omega_{\delta}(m_{\eps}, \bar{\psi}\rho_{\tau})$ appearing in the integral expression. These consist primarily of Gauss sums, see Theorem \ref{eisenalg}, so note that
\[ 
	G(\eta)G(\eta^{n-1}) = G(\bar{\psi})G(\bar{\psi}^{n-1})\frac{G(\rho_{\tau})^n}{J(\rho_{\tau}, \rho_{\tau})^{n-1}J(\bar{\psi}, \rho_{\tau})J(\bar{\psi}^{n-1}, \rho_{\tau}^{n-1})},
\]
where $J(\chi_1, \chi_2) = \sum_{a (mod c)} \chi_1^*(a)\chi_2^*(1-a)$ is the Jacobi sum of any two Hecke characters $\chi_1, \chi_2$ modulo $c\mathbb{Z}$. In the case of $n$ odd and $s>n$ this gives
\[
	\omega_{\delta}(\bar{\psi}) = \omega_{\delta}(\eta)G(\rho_{\tau})^{-n}J(\rho_{\tau}, \rho_{\tau})^{n-1}J(\bar{\psi}, \rho_{\tau})J(\bar{\psi}^{n-1}, \rho_{\tau}^{n-1}). \]
The other cases are simpler and give $\omega_{\delta}(\bar{\psi}) = \omega_{\delta}(\eta)G(\rho_{\tau})^{-n}J(\bar{\psi}, \rho_{\tau})^n$. Denote these products of Jacobi sums by $\mathcal{J}_n(\eta)$ and evidently $\mathcal{J}_n(\eta)^{\sigma} = \mathcal{J}_n(\eta^{\sigma})$. Also, $L_{\psi}(s, f, 1_{\mfrak{m}}) = R_{\psi}(s, f)L_{\psi}(s, f)$, where $R_{\psi}(s, f) = \prod_{p\mid\mfrak{m}}L_p(\psi^{\mfrak{c}}(p)p^{-s})$, and put $P_{\psi}(m_{\eps}) = \prod_{p\in\mathbf{ b}}g_p(\psi^{\mfrak{c}}(p)p^{-m_{\eps}})$. The integral expression (\ref{projintexp}) thus becomes
\begin{align*}
	\frac{\langle f, K(m_{\eps}, \theta_{\mfrak{m}})\rangle}{\mu'(\Lambda, k, \psi)}&\in c_f(\tau, 1)\left[\frac{i^{\frac{n}{2}}|2\tau|^{\frac{1}{2}}}{G(\rho_{\tau})}\right]^n
\frac{\mathcal{J}_n(\eta)P_{\psi}(m_{\eps})}{\Lambda_{\mfrak{c}, \mfrak{y}}(m_{\eps}, \eta)}\mathbb{Q},
\end{align*}
the $\sigma$-equivariance of which is evident. So, for all $\sigma\in\Aut(\mathbb{C})$, we have
\[
	 \left[\frac{\langle f, K(m_{\eps}, \theta_{\mfrak{m}})\rangle}{\mu'(\Lambda, k, \psi)}\right]^{\sigma}
= \frac{\langle f^{\sigma}, K(m_{\eps}, \theta_{\mfrak{m}}^{\sigma})\rangle}{\mu'(\Lambda_{\sigma}, k, \psi^{\sigma})}. 
\]
For any congruence subgroup $\Gamma = \Gamma[\mfrak{x}^{-1}, \mfrak{xy}]$ let 
\[
	\Gamma^0 := \{\gamma\in\Gamma\mid a_{\gamma}\equiv d_{\gamma}\equiv 1\pmod{\mfrak{y}}\}.
\]
Suppose that $\Gamma_2\leq\Gamma_1\leq\mfrak{M}$ are two congruence subgroups, then decompose $\Gamma_1 = \bigsqcup_{i=1}^d\Gamma_2^0 \gamma_i$. The trace map is defined, for a Hecke character $\rho$, by
\begin{align*}
	\Tr_{\Gamma_1, \rho}^{\Gamma_2}:\mathcal{M}_k(\Gamma_2^0)
&\to\mathcal{M}_k(\Gamma_1, \rho) \\
	h&\mapsto\sum_{i=1}^d\rho_{\mfrak{c}_1}(|a_{\gamma_i}|)^{-1}h||_k\gamma_i,
\end{align*}
where $\Gamma_1 = \Gamma[\mfrak{b}_1^{-1}, \mfrak{b}_1\mfrak{c}_1]$.
By Lemma 5.4 of \cite{Bouganis} we have that 
$\Tr_{\Gamma_1, \rho}^{\Gamma_2}(h)^{\sigma} = 
\Tr_{\Gamma_1, \rho^{\sigma}}^{\Gamma_2}(h^{\sigma})$ for any $h\in\mathcal{M}_k(\Gamma_2^0)$ 
and any $\sigma\in\Aut(\mathbb{C})$ 

Therefore, if $\Gamma' = \Gamma[\mfrak{b}^{-1}, \mfrak{bc}']$ then $\Tr_{\Gamma, \psi}^{\Gamma'}(K(m_{\eps}, \theta_{\mfrak{m}}))\in\mathcal{S}_k(\Gamma, \psi)$. So far we have obtained
\[ 
	[\langle f, g\rangle\mu'(\Lambda, k, \psi)^{-1}]^{\sigma} 
= \langle f^{\sigma}, g^{\sigma}\rangle\mu'(\Lambda_{\sigma}, k, \psi^{\sigma})^{-1} 
\]
for all $g\in\{\Tr_{\Gamma_1, \psi}^{\Gamma_2}(K(m_{\eps}, \theta_{\mfrak{m}}))
\mid 0<\tau\in S_{\mfrak{b}}^{\triangledown}, \mfrak{m} =\mfrak{tc}\}$. The rest of 
the proof follows just as in \cite[p. 350]{Sturm} -- 
by extending the above set of $g$ into a basis for 
$\mathcal{S}_k(\Gamma, \psi, \Lambda)$ and using the orthogonal decomposition of $\mathcal{S}_k(\Gamma, \psi)$ into such eigenspaces.
\end{proof}

\begin{theorem}\label{main}
Let $f\in\mathcal{S}_k(\Gamma, \psi, \Lambda)$ be an eigenform for a half-integral weight $k$, congruence subgroup $\Gamma = \Gamma[\mfrak{b}^{-1}, \mfrak{bc}]$ contained in $\mfrak{M}$, Hecke character $\psi$ satisfying (\ref{char1}) and (\ref{char2}), and system of eigenvalues $\Lambda$. Assume that $k > \frac{7n}{2} + \eps$, where $\eps$ is given by $\psi_{\infty}(x) = \sgn(x_{\infty})^{[k]+\eps}$. Let $\chi$ be a Hecke character, and choose $\mu\in\{0, 1\}$ such that $(\psi\chi)_{\infty}(x) = \sgn(x_{\infty})^{[k]+\mu}$ and put $\eta = \psi\chi\rho_{\tau}$. Define the set
\[
	\Omega_{n, k}' := \left\{m\in\tfrac{1}{2}\mathbb{Z}\big|\tfrac{k-m-\mu}{2}\in\mathbb{Z}, \tfrac{3n}{2}+1<m<k-2n-\mu\right\}.
\]
Now if $\tau\in S_+$ is such that $c_f(\tau, 1)\neq 0$ and $m\in\Omega_{n, k}'$, then define
\[ 
	Y_{\psi}(m, f, \chi):=|\tau|^{\frac{\delta}{2}}
\pi^{-b}\mu'(\Lambda, k, \psi)^{-1}\omega_{\delta}(m, \bar{\eta})^{-1}L_{\psi}(m, f, \chi),
\]
where $b = \beta_m + n(k-r) - \frac{n^2+\delta}{4}$, $\ell = \frac{n}{2}+\delta$, and recall $\delta = n\pmod{2}\in\{0, 1\}$. We have $Y_{\psi}(m, f, \chi)^{\sigma} = 
Y_{\psi^{\sigma}}(m, f^{\sigma}, \chi^{\sigma})$ 
for all $\sigma\in\Aut(\mathbb{C}/\mathbb{Q})$ and hence
\[
	Y_{\psi}(m, f, \chi)\in\mathbb{Q}(f, \psi, \chi).
\]
\end{theorem}

\begin{proof}
Noting that $|\tau|^{\frac{\delta}{2}}\in (2|2\tau|)^{\frac{\delta}{2}}
\mathbb{Q}$ and $\Omega'_{\theta_{\chi}}= \Omega_{n, k}'$, we use the integral expression (\ref{projintexp}) to get
\[ 
	Y_{\psi}(m, f, \chi) \in c_f(\tau, 1)^{-1}\Lambda_{\mfrak{c}, \mfrak{y}}(m, \eta)
\prod_{p\in\mathbf{b}}g_p(\left(\psi^{\mfrak{c}}\chi^*)(p)p^{-m}\right)
\frac{\langle f, K(m, \theta_{\chi})\rangle}{\mu'(\Lambda, k, \psi)}\mathbb{Q},
\]
which, once we consider Theorem \ref{innerprod}, we see is 
$\sigma$-equivariant for all of $\Aut(\mathbb{C}/\mathbb{Q})$.
\end{proof}

\section{Algebraicity of metaplectic Eisenstein series}
\label{eisensteinsection}

So far we have determined the specific algebraicity for only some of the special values given in \cite{Shimurabook}. The aim of this lengthy section is to investigate and establish the precise algebraicity of the well-known decomposition $\mathcal{M}_k = \mathcal{S}_k\oplus\mathcal{E}_k$ which will allow the determination of the rest of these special values. Such a algebraicity is equivalent to proving Garrett's conjecture that the Klingen Eisenstein series $E(f)$ of a cusp form $f$ preserves algebraicity, see \cite{Garrett}. 

Due to its length, this section is split up into two subsections. The first is preliminary and the main aim is to relate the standard $L$-function of $\Phi f$ with $f$ where $\Phi$ is the Siegel Phi operator. This relation will be useful in the second subsection, which is the proof of Garrett's conjecture and the desired decomposition.
 
\subsection{Hecke eigenforms and the Siegel Phi operator}
\label{eisensteinsubsect1}

Here we give a relation between the $L$-function of an $n$-degree form $f$ and that of the $n-1$-degree form $\Phi f$, where $\Phi$ is the Siegel Phi operator. Such a relation has been studied before: for integral-weight forms this was done by Zharkovskaya in \cite{Zha} and later for non-trivial character by Andrianov in \cite{Andrianov}. It has also been established by Hayashida \cite{hayashida} for half-integral weight Siegel modular forms, using the Hecke ring construction of Zhuravlev in \cite{Zhu1}, \cite{Zhu2} and the results of Oh-Koo-Kim \cite{OOK}, but it is not clear how their setting translates to that of the present.

For a real variable $\rho$ the Siegel Phi operator is defined as
\begin{align*}
	\Phi:\mathcal{M}_k^n &\to \mathcal{M}_k^{n-1} \\
	f(z)&\mapsto \lim_{\rho\to\infty} f\begin{pmatrix} w & 0 \\ 0 & i\rho\end{pmatrix},
\end{align*}
for $z\in\mathbb{H}_n, w\in \mathbb{H}_{n-1}$. 

In order to establish the desired relation of $L$-functions of $\Phi f$ and $f$ we need to relate their Satake parameters. Define
\begin{align*}
	\Sx_p &:= M_n(\mathbb{Z}_p)\cap GL_n(\mathbb{Q}_p), &&\Sx:= GL_n(\mathbb{Q})_{\mathbf{f}}\prod_p \Sx_p, \\
	\So_p&:= GL_n(\mathbb{Z}_p), && \So:=\prod_p \So_p, \\
	Z_0&:=\{\diag[\tilde{q}, q]\mid q\in \Sx\}, && Z:= D[2, 2]Z_0D[2, 2],
\end{align*}
and certain metaplectic lifts
\begin{align*}
	\mfrak{D}[2, 2] &:= \pr^{-1}(D[2, 2]), && \mfrak{D} := \{\alpha\in\mfrak{D}[2, 2]\mid \pr(\alpha)\in G_{\mathbf{f}}\cap D\}, \\
	\mfrak{Z} &:= \pr^{-1}(Z), && \mfrak{Z}_0 := \{\alpha\in\mfrak{Z}\mid \pr(\alpha)\in G_{\mathbf{f}}\cap DZ_0D\}, \\
	\widehat{\mfrak{Z}}_0 &:= \{(\alpha, t)\mid t\in\mathbb{T}, \alpha\in\mfrak{Z}_0\}, && \widehat{\mfrak{D}} := \{(\alpha, 1)\in \widehat{\mfrak{Z}}_0\mid \alpha\in\mfrak{D}\},
\end{align*}
where $D = D[\mfrak{b}^{-1}, \mfrak{bc}]$ for ideals $(\mfrak{b}^{-1}, \mfrak{bc})\subseteq 2\mathbb{Z}\times 2\mathbb{Z}$. For any prime $p$ we use the subscript $p$ to denote the $p$th local component of any of the above adelic groups. All of the above sets are dependent on $n$ and, 
when we wish to distinguish this, we shall use $n$ as a superscript 
e.g. $\Sx_p^n$. The abstract Hecke ring $\mathcal{R}(\widehat{\mfrak{D}}, \widehat{\mfrak{Z}}_0)$ comprises all formal, finite sums 
\[ 
	\sum_{\sigma}c_{\sigma}\widehat{\mfrak{D}}\sigma \widehat{\mfrak{D}},
\]
with $c_{\sigma}\in\mathbb{C}$ and $\sigma\in \widehat{\mfrak{Z}}_0$. 

This abstract Hecke ring has a 
representation on the space of modular forms in 
$\mathcal{M}_k(\Gamma, \psi)$, where $\Gamma = G\cap D$, which we now describe. In \cite[p. 32]{Shimurahalf} Shimura defines a new factor of automorphy $J^k$, which extends the original $j^k$ to $\mfrak{Z}$ and has strong automorphic properties.
Now consider the element $\widehat{\mfrak{D}}(r_P(\sigma), 1)\widehat{\mfrak{D}}\in\mathcal{R}(\widehat{\mfrak{D}}, \widehat{\mfrak{Z}}_0)$,
where $\sigma = \diag[\tilde{q}, q]$ for $q\in \Sx$, then we have 
the decomposition $G\cap(D\sigma D) = \Gamma\xi\Gamma = \bigsqcup_{\alpha}
\Gamma\alpha$ for some $\alpha, \xi\in G\cap Z$. The 
representation of this element, denoted $T_{q, \psi}$, on $f\in\mathcal{M}_k(\Gamma, \psi)$ 
is given by
\[ 
	(f|T_{q, \psi})(z) = \sum_{\alpha}\psi_{\mfrak{c}}(|a_{\alpha}|)^{-1} 
J^k(\alpha, z)^{-1}f(\alpha z).
\]
Let
$\mathcal{R}_0$ be the factor ring of $\mathcal{R}(\widehat{\mfrak{D}}, \widehat{\mfrak{Z}}_0)$ modulo the ideal $\langle \widehat{\mfrak{D}}(\alpha, 1)\widehat{\mfrak{D}}-t\widehat{\mfrak{D}}(\alpha, t)\widehat{\mfrak{D}}\mid (\alpha, t)\in \widehat{\mfrak{Z}}_0\rangle$.
Since the action of this ideal is trivial on $\mathcal{M}_k(\Gamma, \psi)$, see \cite[p. 41]{Shimurahalf}, we have an action of this factor ring on modular forms defined as before. 
Then denote the element represented by $\widehat{\mfrak{D}}(r_P(\sigma), 1)\widehat{\mfrak{D}}$ in 
$\mathcal{R}_0$ by $A_q$, where $\sigma=\diag[\tilde{q}, q]$ 
with $q\in \Sx$, that is $A_q$ is the image of $T_{q, \psi}$ in 
$\mathcal{R}_0$. Finally we denote by $\mathcal{R}_{0p}$ the 
subalgebra of $\mathcal{R}_0$ that is generated by the $A_q$ for 
all $q\in \Sx_p$. As in \cite[pp. 41--42]{Shimurahalf}, define a map
\[ 
	\omega_p := \omega_{0p}\circ\Phi_p : 
\mathcal{R}_{0p}\to\mathbb{C}[x_1^{\pm 1}, \dots, x_n^{\pm 1}]
\]
as follows.
If $\sigma = \diag[\tilde{q}, q]$ with $q\in \Sx_p$, then 
$A_q\in \mathcal{R}_{0p}$ and we have a decomposition of the form
\begin{align} 
	D_p\sigma D_p = \bigsqcup_{x\in X}\bigsqcup_{s\in Y_x}
\bigsqcup_{d\in R_x}D_p\alpha_{d, s}, \hspace{20pt} \alpha_{d, s} 
= \begin{pmatrix}\tilde{d} & sd\\ 0 & d\end{pmatrix},  \label{satakecoset}
\end{align}
with $X\subseteq GL_n(\mathbb{Q}_p), R_x\subseteq x\So_p$ 
representing $\So_p\bslsh \So_px\So_p$, and $Y_x\subseteq S_p$. 
Then extend to all of $\mathcal{R}_{0p}$ by $\mathbb{C}$-linearity 
the following map
\[ 
	\Phi_p(A_q) := \sum_{d, s}J(r_P(\alpha_{d, s}))^{-1}\So_pd
\in \mathcal{R}(\So_p, GL_n(\mathbb{Q}_p)),
\]
where $J(\alpha) = J^{\frac{1}{2}}(\alpha, \mathbf{i})$ for $\alpha\in\mfrak{Z}$.
For the second map $\omega_{0p}$, note that any coset $\So_pd$ 
with $d\in GL_n(\mathbb{Q}_p)$ contains an upper triangular 
matrix of the form 
\begin{align} 
	\begin{pmatrix} 
	p^{a_{d_1}} & \star & \cdots & \star \\ 
	0 & p^{a_{d_2}} & \cdots & \star \\ \vdots & \vdots & \ddots & \vdots \\ 
	0 & 0 & \cdots & p^{a_{d_n}}
	\end{pmatrix}, \label{uppertriangular}
\end{align}
with $a_{d_i}\in\mathbb{Z}$, and then define
\[ 
	\omega_{0p}(\So_pd) = \prod_{i=1}^n (p^{-i}x_i)^{a_{d_i}} 
\]
which, via the decomposition $\So_px\So_p = \sum_{d} \So_pd$ 
and $\mathbb{C}$-linearity, we extend to obtain the map $\omega_{0p}:
\mathcal{R}(\So_p, GL_n(\mathbb{Q}_p))\to\mathbb{C}
[x_1^{\pm 1}, \dots, x_n^{\pm 1}]$. 

Assume that $p\nmid\mfrak{c}$. For an independent variable $u$ define a map 
$\Psi(\cdot, u):\mathcal{R}_{0p}^n\to\mathcal{R}_{0p}^{n-1}[u^{\pm 1}]$ as follows. 
Consider the generators $A_q\in \mathcal{R}_{0p}$ for $q\in \Sx_p$, 
with $D_p\sigma D_p$ having the decomposition (\ref{satakecoset}) and each $d$ 
having the form (\ref{uppertriangular}), and put
\[ 
	\Psi(A_q, u)=\sum_{x, d, s}\frac{J(r_P(\alpha_{d', s'}))}{J(r_P(\alpha_{d, s}))}
(up^{-n})^{a_{d_n}}D_p\begin{pmatrix} \tilde{d}' & s'd' \\ 0 & d'\end{pmatrix},
\]
where $A'$ denotes the upper left $n-1$ block of $A\in M_n$. Extend this to all of $\mathcal{R}_{0p}$ by $\mathbb{C}$-linearity. 
The map $\omega_p^{n-1}\times 1$ acts as $\omega_p^{n-1}$ on 
$\mathcal{R}_{0p}^{n-1}$ and as the identity on $u$. Then
\[ 
	(\omega_p^{n-1}\times 1)(\Psi(A_q, u)) 
= \sum_{d, s}J(r_P(\alpha_{d, s}))^{-1}(up^{-n})^{a_{d_n}}
\prod_{i=1}^{n-1}(p^{-i}x_i)^{a_{d_i}}. 
\]
So by defining $\phi_{n, u}(x_i) = x_i$ for $1\leq i\leq n-1$ and 
$\phi_{n, u}(x_n) = u$, extending $\mathbb{C}$-linearly to all of 
$\mathbb{C}[x_1^{\pm 1}, \dots, x_n^{\pm 1}]$, we get the commuting square
\begin{equation}
\begin{tikzcd}
	\mathcal{R}_{0p}^n  \arrow[d, "\Psi({\cdot, u})"] \arrow[r, "\omega_p^n"]
		& \mathbb{C}[x_1^{\pm 1}, \dots, x_n^{\pm 1}] \arrow[d, "\phi_{n, u}"] \\
	\mathcal{R}_{0p}^{n-1}[u^{\pm 1}] \arrow[r, "\omega_p^{n-1}\times 1"] 
		& \mathbb{C}[x_1^{\pm 1}, \dots, x_{n-1}^{\pm 1}, u^{\pm 1}]
\end{tikzcd} \label{commsq}
\end{equation} 
 By 
\cite[Lemma 2.6]{Shimuraint} we are able to choose $\alpha_{d, s}$ and $\alpha_{d', s'}$
of the form
\[
\alpha_{d, s} = \begin{pmatrix} g^{-1}h & g^{-1}\sigma\tilde{h} \\ 0 & g^T\tilde{h}\end{pmatrix}, \hspace{20pt} \alpha_{d', s'} = \begin{pmatrix} (g')^{-1}h' & (g')^{-1}\sigma'\tilde{h}' \\ 0 & (g')^T\tilde{h}'\end{pmatrix},
\]
where $g, h\in \Sx_p^n$, and $\sigma\in gS_{\mathbf{f}}(\mfrak{b}^{-1})g^T$. We have $J^k(r_P(\alpha), z) = J^{\frac{1}{2}}(r_P(\alpha), z)j(\alpha, z)^{[k]}$ by (2.1c) of \cite{Shimurahalf}, and by Lemma 2.4 of that same paper  $J^{\frac{1}{2}}(\alpha_{d, s}, z)$ is independent of $z$. So we see that
\begin{align}
	J^k(\alpha_{d, s}, z) = J(r_P(\alpha_{d, s}))J(r_P(\alpha_{d', s'}))^{-1}p^{a_{d_n}[k]}J^k(\alpha_{d', s'}, z). \label{Jidentity}
\end{align}

\begin{proposition}\label{fouriercalc} 
If $f\in\mathcal{M}_k(\Gamma, \psi)$ and $p\nmid\mfrak{c}$ then
\[ 
	\Phi(f|A_q) =
	(\Phi f)|\Psi(A_q, \psi_p^{-1}(p)p^{n-[k]}).
\] 
\end{proposition}

\begin{proof} Firstly, since $\psi_{\mfrak{c}}(p) = \psi_p^{-1}(p)$, we have
\[ 
	f|A_q = \sum_{x, d, s}\sum_{\tau\in S_+}\psi_{p}^{-1}(|d|)J^k(\alpha_{d, s}, z)^{-1}
c_f(\tau, 1)e(d^{-1}\tau\tilde{d}\cdot z + \tau s),
\]
and apply $\Phi$ to the above expression. If 
$\tau = \begin{psmallmatrix} \tau' & \star \\ \star & t\end{psmallmatrix}$ 
with $\tau'\in S_+^{n-1}$ and $0\leq t\in\mathbb{Z}$, then we know 
that the last diagonal entry of $d^{-1}\tau\tilde{d}$ is $p^{-2a_{d_n}}t$. 
Thus by writing $z=\begin{psmallmatrix} z' & 0 \\ 0 & i\lambda\end{psmallmatrix}$ 
and letting $\lambda\to\infty$, any terms involving $\tau$ for $t>0$ 
will tend to 0 and we are left only with terms involving $\tau\in S_+^n$ 
with $t=0$ -- these are precisely the elements of $S^{n-1}_+$. So for 
$\tau =\begin{psmallmatrix} \tau' & 0 \\ 0 & 0\end{psmallmatrix}$ 
we have
\[ 
	d^{-1}\tau\tilde{d}\cdot z + \tau s = \begin{pmatrix} (d')^{-1}
\tau'\tilde{d}'\cdot z' + \tau's' & 0 \\ 0 & 0 \end{pmatrix}. 
\]
By the relation between $J^k(\alpha_{d, s}, z)$ and $J^k(\alpha_{d', s'}, z)$ of (\ref{Jidentity}), we get
\begin{align*}
	\Phi(f|A_q) &= \sum_{x, d, s}\frac{J(r_P(\alpha_{d', s'}))}{J(r_P(\alpha_{d, s}))}\psi_p^{-1}(p^{a_{d_n}})p^{-a_{d_n}[k]} \\
&\times \left[\psi_{\mfrak{c}}(|d'|)J^k(\alpha_{d', s'}, z')^{-1}\sum_{\tau'\in S_+^{n-1}} 
c_f\left(\begin{pmatrix} \tau' & 0 \\ 0 & 0 \end{pmatrix}, 1\right)
e(\tau'(\tilde{d}'\cdot z' + s'd')(d')^{-1})\right],
\end{align*}
which is exactly $(\Phi f)|\Psi(A_q, \psi_p^{-1}(p)p^{n-[k]})$, as $\Phi f$ has Fourier coefficients $c_f\left(
\begin{psmallmatrix} \tau' & 0 \\ 0 & 0\end{psmallmatrix}, 1\right)$ 
for all $\tau'\in S_+^{n-1}$. 
\end{proof}

For each prime $p$ an element of $\mathbb{C}^n$ -- the Satake $p$-parameters -- is associated 
to a Hecke eigenform in $\mathcal{S}_k^n$ and this process, taken from \cite{Shimurahalf}, we outline briefly. In doing so, we are able to see how the elements 
of $\mathbb{C}^{n-1}$ and $\mathbb{C}^n$ for $\Phi f$ and $f$, 
respectively, are related. 

Assume that $p\nmid\mfrak{c}$. Define the operator
\[ 
	\mathcal{T}_p^n : = \sum_{m=0}^{\infty}A_{\psi}^n(p^m) t^m,
\]
where $A_{\psi}(p^m)$ is the sum of all $A_q$ with $|q| = p^m$. 
If $f$ be a Hecke 
eigenform of degree $n$ with $f|A_{\psi}^n(p^m) = \Lambda(p^m)f$ then
\begin{align}
	\Phi(f|\mathcal{T}_p^n) = \sum_{m=0}^{\infty} \Lambda(p^m)t^m \Phi f.  \label{eigenvalues}
\end{align}

Extend the 
definitions of $\Psi, \omega_p^n$, and $\omega_p^{n-1}$ to 
$\mathcal{T}_p$ by letting them act linearly on the coefficients. For any $1\leq \ell\in\mathbb{Z}$,
Theorem 4.4 in 
\cite[p. 42]{Shimurahalf} gives
\begin{align*}
\begin{split}
	\omega_p^{\ell}(\mathcal{T}_p^{\ell}) =
	\displaystyle\prod_{i=1}^{\ell} \frac{1-p^{2i-1}t^2}{(1-p^{\ell}x_it)(1-p^{\ell}x_i^{-1} t)} 
\end{split}
\end{align*}
and hence, for an $\ell$-degree eigenform $g$, the existence of the Satake $p$-parameters 
$(\lambda_{p, 1}, \dots, \lambda_{p, {\ell}})$
such that 
\[ 
	g|\mathcal{T}_p^{\ell} = 
	\displaystyle\prod_{i=1}^{\ell}\frac{1-p^{2i-1}t^2}
{(1-p^{\ell}\lambda_{p, i}t)(1-p^{\ell}\lambda_{p, i}^{-1}t)}g.
\]

Assume that $0\neq \Phi f$ has Satake $p$-parameters $(\lambda_{p,1}, \dots, \lambda_{p, n-1})$ for $p\nmid\mfrak{c}$. By the commuting square in (\ref{commsq}) we have
\begin{align*}
	\omega_p^{n-1}(\Psi(\mathcal{T}_p^n, u)) &=\phi_{n, u}(\omega_p^n(\mathcal{T}_p^n))\\
	& = \left[\prod_{i=1}^{n-1}\frac{1-p^{2i-1}t^2}{(1-p^nx_it)(1-p^nx_i^{-1}t)}\right] \frac{1-p^{2n-1}t^2}{(1-p^n ut)(1-p^n u^{-1}t)},
\end{align*}
so that
\begin{align}
	(\Phi f)|\Psi(\mathcal{T}_p^n, u) = \left[\prod_{i=1}^{n-1}\frac{1-p^{2i-1}t^2}{(1-p^n\lambda_{p, i}t)(1-p^n\lambda_{p, i}^{-1}t)}\right]\frac{1-p^{2n-1}t^2}{(1-p^nut)(1-p^nu^{-1}t)} \Phi f. \label{Psi1}
\end{align}
On the other hand Proposition \ref{fouriercalc} along with the identity in (\ref{eigenvalues}) above gives
\begin{align}
	 (\Phi f)|\Psi(\mathcal{T}_p, \psi_p^{-1}(p)p^{n-[k]}) = \Phi(f|\mathcal{T}_p^n) = \sum_{m=0}^{\infty}\Lambda(p^m)t^m\Phi f.  \label{Psi2} 
\end{align}
So, equating (\ref{Psi1}) and (\ref{Psi2}) with $u = \psi_p^{-1}(p)p^{n-[k]}$ we have proved the following.

\begin{proposition} \label{parameters} Let $f\in\mathcal{M}_k^n(\Gamma, \psi)$ be a non-zero eigenform such that $\Phi f\neq0$. Then $\Phi f$ is an eigenform of degree $n -1$. If $\Phi f$ has Satake $p$-parameters $(\lambda_{p, 1}, \dots, \lambda_{p, n-1})$ for $p\nmid\mfrak{c}$ then the Satake $p$-parameters of $f$ are  $(\lambda_{p, 1}, 
\dots, \lambda_{p, n-1}, \psi_p^{-1}(p)p^{n-[k]})$.
\end{proposition}

Define the Hecke character $\chi := \psi^{-2}$. We can use the above Proposition \ref{parameters} to obtain a relation between $L_{\psi}^n(s, f, \chi)$ and $L_{\psi}^{n-1}(s-1, \Phi f, \chi)$. Assume that $\Phi f\neq 0$ has Satake $p$-parameters $(\lambda_{p, 1}, \dots, \lambda_{p, n-1})$ then, by Proposition \ref{parameters} above, the local Euler factor of $f$ at $p\nmid\mfrak{c}$ is
\[
	L_p^n\left((\psi^{\mfrak{c}}\chi^*)(p)p^{-s}\right) = 
	L_p^{n-1}\left((\psi^{\mfrak{c}}\chi^*)(p)p^{-s+1}\right)(1-\chi^*(p)p^{2n-[k]-s})(1-p^{[k]-s}),
\]
and the Euler factors at $p\mid\mfrak{c}$ are just 1 by definition of $\chi$. Therefore
\begin{align*}
	L_{\psi}^n(s, f, \chi) = L_{\psi}^{n-1}(s-1, \Phi f, \chi)L(s+[k]-2n,\chi)\zeta_{\mfrak{c}}(s-[k]),
\end{align*}
where $\zeta_{\mfrak{c}}$ is the Riemann zeta function with the Euler factors at $p\mid\mfrak{c}$ removed. By induction, for any $0\leq r'\leq n$ such that $\Phi^{n-r'}f\neq 0$, we get
\begin{align}
	L_{\psi}^n(s, f, \chi) = L_{\psi}^{r'}(s-n+r', \Phi^{n-r'}f, \chi)\prod_{i=0}^{n-r'-1}L(s+[k]-2n+i, \chi)\zeta_{\mfrak{c}}(s-[k]-i).  \label{PhiLfun}
\end{align}

\subsection{Klingen Eisenstein series}

Let $G'$ denote the image of $G$ under the embedding
\begin{align*}
	G&\to G_{\mathbb{A}} \\ x&\mapsto (x_v)_v,
\end{align*}
where $x_{\infty} = x$ and $x_p = I_{2n}$ for all primes $p$. Let $\mfrak{G} = \pr^{-1}(G')\leq M_{\mathbb{A}}$ and we have $\mfrak{G}\leq \mfrak{M}$. By \cite[p. 554]{Shimuraeis} the group $\mfrak{G}$ can be identified with the group of couples $(\alpha, q)$ where 
$\alpha\in G$ and $q:\mathbb{H}_n\to\mathbb{C}$ is a holomorphic 
function such that $q(z)^2/j_{\alpha}(z)\in\mathbb{T}$ is a constant, 
with group law $(\alpha, q)(\alpha', q') = (\alpha\alpha', q(\alpha'z)q'(z))$. This identification is given by $\alpha\mapsto (\alpha, h_{\alpha})$ and $\mfrak{G}$ acts on $f:\mathbb{H}_n\to\mathbb{C}$ as
\[
	(f||_k\xi)(z) := q(z)j(\alpha, z)^{[k]}f(\alpha z),
\]
where $\xi = (\alpha, q)\in\mfrak{G}$.

We have previously been 
considering congruence subgroups $\Gamma$ of $G$ that are contained 
in $\mfrak{M}$, and to such a congruence subgroup we define the group 
$\widehat{\Gamma} = \{(\alpha, h_{\alpha})\mid \alpha\in \Gamma\}\leq \mfrak{G}$. 
Indeed, the definition of a congruence subgroup of $\mfrak{G}$ is 
given in such a way -- it is a subgroup $\Delta\leq\mfrak{G}$ that is 
isomorphic to a congruence subgroup $\Gamma\leq G$ 
via $\Delta = \widehat{\Gamma}$. As such, congruence subgroups 
of $G$ and $\mfrak{G}$ are one and the same and we simply use $\Gamma$ to denote a congruence subgroup of $\mfrak{G}$ as well.

For an integer $r$ such that $0\leq r\leq n$ and for any $\alpha\in 
M_n(\mathbb{A}_{\mathbb{Q}})$ we write
\begin{align} 
	\alpha = \begin{pmatrix}\begin{pmatrix} a_1 & a_2 \\ a_3 & a_4\end{pmatrix} 
& \begin{pmatrix} b_1 & b_2 \\ b_3 & b_4\end{pmatrix} \\ 
	\begin{pmatrix} c_1 & c_2 \\ c_3 & c_4\end{pmatrix} 
& \begin{pmatrix} d_1 & d_2 \\ d_3 & d_4\end{pmatrix}\end{pmatrix}, \label{alpha}
\end{align}
where, for $x\in\{a, b, c, d\}$, we have $x_1\in M_r(\AAQ), x_2\in 
M_{r, n-r}(\AAQ), x_3\in M_{n-r, r}(\AAQ)$, and $x_4\in M_{n-r}(\AAQ)$. Also write
\[ 
	x_{\alpha} = \begin{pmatrix} x_1 & x_2 \\ x_3 & x_4\end{pmatrix} 
= \begin{pmatrix} x_1(\alpha) & x_2(\alpha) \\ x_3(\alpha) & x_4(\alpha)
\end{pmatrix} 
\]
when we wish to emphasis the matrix $\alpha$ to which these blocks 
belong. If $r=n$ then we make the natural understanding that 
$x_{\alpha} = x_1(\alpha)$ and likewise, for $r=0$, we have 
$x_{\alpha} = x_4(\alpha)$. Now for such an $n, r$ we define 
the following parabolic subgroup $P^{n, r}\leq Sp_n(\mathbb{Q})$ by
\begin{align*}
	P^{n, r} := \{\alpha\in Sp_n(\mathbb{Q})\mid a_2(\alpha) 
= c_2(\alpha) = 0, c_3(\alpha) = d_3(\alpha) = 0, c_4(\alpha) = 0\},
\hspace{10pt}\text{if $0<r<n$},
\end{align*}
with $P^{n, 0} = P^n$ and $P^{n, n} = Sp_n(\mathbb{Q})$. 
With $\alpha$ of the form in (\ref{alpha}) we have some maps:
\begin{align*}
	\pi_r:M_{2n}(\AAQ)&\to M_{2r}(\AAQ) \\
	\alpha&\mapsto \begin{pmatrix} a_1(\alpha) 
& b_1(\alpha) \\ c_1(\alpha) & d_1(\alpha)\end{pmatrix}; \\
	\lambda_r:M_{2n}(\AAQ) &\to \AAQ \\
	\alpha&\mapsto |d_4(\alpha)|.
\end{align*}
These define respective homomorphisms 
$P_{\mathbb{A}}^{n, r}\to Sp_r(\AAQ)$ and 
$P_{\mathbb{A}}^{n, r}\to\mathbb{I}_{\mathbb{Q}}$. 
On the metaplectic side let $\mfrak{P}^{n, r} := \{(\alpha, p)\in\mfrak{G}\mid\alpha\in P^{n, r}\}$ 
and extend $\pi_r, \lambda_r$ to $\mfrak{P}^{n, r}$ by letting
\begin{align*} 
	\pi_r((\alpha, q)) &:= (\pi_r(\alpha), |\lambda_r(\alpha)|^{-\frac{1}{2}}q')
\in\mfrak{G}^r, \\
	\lambda_r((\alpha, q)) &:= \lambda_r(\alpha_{\infty})\in\mathbb{Q}^{\times},
\end{align*}
where $q'(z) = q\left(\begin{psmallmatrix} z & w \\ w^T & z'\end{psmallmatrix}\right)$ 
does not depend on the choice of $w, z'$.

Supposing $0\leq r\leq n$, $\Gamma\leq \mfrak{G}^n$ is a 
congruence subgroup, $\psi$ is a Hecke character, and $K$ is some number field, then for $\mathcal{X}\in\{\mathcal{M}, \mathcal{S}\}$ 
we denote
\[ 
	\mathcal{X}_k^r(\Gamma\cap\mfrak{P}^{n, r}, \psi) := 
\{f\in \mathcal{X}_k^r\mid f||_k\pi_r(\gamma) = (\sgn^{[k]}\psi_{\mfrak{c}}^{-1})(\lambda_r(\gamma))
f\ \text{for all}\ \gamma\in \Gamma\cap\mfrak{P}^{n, r}\},
\]
and by $\mathcal{X}_k^r(\Gamma\cap\mfrak{P}^{n, r}, \psi, K)$ the forms 
of the above set with coefficients in $K$.
We are now ready to define a certain class of Eisenstein series, 
the so-called \emph{Klingen Eisenstein series}. If $0\leq r\leq n$ and 
$f\in\mathcal{S}_k^r(\Gamma\cap\mfrak{P}^{n, r}, \psi)$ then in 
\cite[pp. 547, 554]{Shimuraeis} these Eisenstein series are defined as
\begin{align*}
	E_k^{n, r} (z; f, \psi, \Gamma)
:= \sum_{\gamma\in(\Gamma\cap\mfrak{P}^{n, r})\bslsh\Gamma} \psi_{\mfrak{c}}(|a_{\gamma}|)f\left(z^{(r)}\right)\big|\big|_k\gamma,
\end{align*}
where $z^{(r)}$ is the upper left $r\times r$ block of $z$, 
and this is convergent provided $k>n+r+1$. In \cite[p. 356]{Shimuraexp} 
this was extended to all $k>\frac{n+r+3}{2}$. Put $E_k^{n, r}(z; f, \Gamma) := E_k^{n, r}(z; f, 1, \Gamma)$. By Lemma 8.11 of \cite{Shimuraeis} these are holomorphic if $k>2n$.

Assume that $k>2n$. The span of all such Eisenstein series is a space that will 
play a large role in this section and is denoted
\begin{align*} 
	\mathcal{E}_k^{n, r} := \spac\{E_k^{n, r}(z; f, \Gamma)||_k\alpha
\mid \alpha\in \mfrak{G}^n, f\in\mathcal{S}_k^r(\Gamma\cap\mfrak{P}^{n, r}, 1), \Gamma\leq\mfrak{G}^n\}.
\end{align*}
As $E^{n, n}_k(z; f, \Gamma) = f$, we have $\mathcal{E}_k^{n, n} = \mathcal{S}_k^n$. 
Set $\mathcal{E}_k^{n, r}(\Gamma, \psi) 
= \mathcal{E}_k^{n, r}\cap\mathcal{M}_k^n(\Gamma, \psi)$ for any congruence subgroup $\Gamma$ and Hecke character $\psi$. Their 
eminence in this section comes from a convenient decomposition, 
in their terms, of the space of modular forms.

\begin{theorem}[(Shimura, \cite{Shimuraeis}, pp.581--582)]\label{decomp} 
Let $k>2n$ be a half-integral weight. Then 
we have the decompositions
\begin{align*}
	\mathcal{M}_k^n &= \bigoplus_{r=0}^n \mathcal{E}_k^{n, r}, \\
	\mathcal{M}_k^n(\Gamma, \psi) &= \bigoplus_{r=0}^n \mathcal{E}_k^{n, r}(\Gamma, \psi).
\end{align*}
\end{theorem}

\begin{remark} 
Originally in \cite{Shimuraeis} the above theorem was 
proven for the bound $k> 2n$; this bound was further improved in 
[\cite{Shimuraexp}, p.346]. We retain those of the former as later results 
will require us to take this bound regardless.
\end{remark}

For any integral ideal $\mfrak{a}$ we let $\mathcal{R}_0^{\mfrak{a}}$ 
(resp. $\mathcal{R}_{\mfrak{a}}(\widehat{\mfrak{D}}, \widehat{\mfrak{Z}}_0)$) denote the subspace of $\mathcal{R}_0$ 
(resp. $\mathcal{R}(\widehat{\mfrak{D}}, \widehat{\mfrak{Z}}_0)$) generated by all $A_q$ (resp. $T_{q, \psi}$) 
with $q_p\in \Sx_p$ for all $p$ and $q_p\in \So_p$ if $p\mid\mfrak{a}$.

\begin{theorem}\label{harris} 
Let $r, r'\in\mathbb{Z}$, $0\leq r'\leq r\leq n$, and assume
$[k]>\frac{n}{2}+r'+1$. Consider two non-zero Hecke eigenforms $f\in \mathcal{E}_k^{n, r}
(\Gamma, \psi)$, $f'\in \mathcal{E}_k^{n, r'}(\Gamma, \psi)$ with the same eigenvalues for $\mathcal{R}_0^{\mfrak{c}}$. Then $r=r'$.
\end{theorem}

\begin{proof}
As in \cite[p. 309]{Harris} we may assume that $r=n$ and therefore that
$f$ is a cusp form. Assume for a contradiction that $r'<r$. Since $f$ and $f'$ 
share the same eigenvalues for $\mathcal{R}_0^{\mfrak{c}}$ we have $L_{\psi}^n(s, f, \chi) = L_{\psi}^n(s, f', \chi)$,  where we have set $\chi := \psi^{-2}$. The relation in (\ref{PhiLfun}) obtained at the end of the last subsection then gives
\begin{align*}
 	L_{\psi}^n(s, f, \chi) = L_{\psi}^{r'}(s-n+r', \Phi^{n-r'}f', \chi)\prod_{i=0}^{n-r'-1}L(s+[k]-2n+i, \chi)\zeta_{\mfrak{c}}(s-[k]-i).
\end{align*}  
Plug $s=[k] +n -r'$ into this. For $i=n-r'-1$ 
we have that $\zeta_{\mfrak{c}}(s-[k] -i) = \zeta_{\mfrak{c}}(1)$ is a pole; $\zeta_{\mfrak{c}}(s-[k]- i) \neq 0$ for all other $i$ and 
$L(s+[k]-2n+i, \eta)\neq 0$ for all $i$. By Theorem A in \cite[p. 332]{Shimuraexp} $L_{\psi}^n(s, f, \eta)$ is 
absolutely convergent for $\Re(s)>\frac{3n}{2}+1$ and by our choice 
of $s$ and $k$ we indeed have this. So the left-hand side is \emph{finite}. 
In the same manner $L_{\psi}^{r'}(s', \Phi^{n-r'}f', \eta)$ is absolutely convergent for all 
$\Re(s') >\frac{3r'}{2}+1$, which inequality $s' = s-n+r'$ satisfies by 
our choice of $s$ and $k$. Thus $L_{\psi}^{r'}(s-n+r', \Phi^{n-r'}f', \eta)$ is \emph{non-zero}. 
We arrive then at a contradiction as the right-hand side of this 
expression contains a pole, yet the left does not. So $r' = r$.
\end{proof}

For any $0\leq r\leq n$ we let $X_r = \mfrak{P}^{n, r}\bslsh
\mfrak{G}^n/\Gamma$ be representatives for the $r$-dimensional cusps. For any $\xi\in X_{n-1}$ and $f\in\mathcal{M}_k^n(\Gamma)$ let
\[
	\Phi_{\xi}f = \Phi(f||_k\xi^{-1}).
\]
Then we define
\begin{align*}
	\Phi_{\star}:\mathcal{M}_k^n(\Gamma, \psi)&\to\prod_{\xi\in X_{n-1}}
\mathcal{M}_k^{n-1}(\xi\Gamma\xi^{-1}\cap\mfrak{P}^{n, n-1}, \psi) \\
	f&\mapsto (\Phi_{\xi}f)_{\xi},
\end{align*}
and by definition $\ker(\Phi_{\star}) = \mathcal{S}_k^n(\Gamma, \psi)$. 

\begin{lemma} \label{heckeslash} 
If $f\in\mathcal{M}_k(\Gamma, \psi)$ then $(f||_k\xi^{-1})|A_q 
= (f|A_q)||_k\xi^{-1}$ for any $\xi\in X_{n-1}$ and 
any $A_q\in\mathcal{R}_0^{\mfrak{c}}$.
\end{lemma}

\begin{proof} 
Since $\xi$ is the identity at finite places we have, for $\sigma = \diag[\tilde{q}, q]$ with $q\in \Sx$ and $q_p\in \So_p$ for all $p\mid\mfrak{c}$, that 
\[
	(\xi D\xi^{-1})\sigma(\xi D\xi^{-1}) = \xi Sp_n(\mathbb{R})\xi^{-1}\times\prod_{p\nmid\mfrak{c}}D_p\sigma_p D_p.
\]
From this $G\cap(\xi D\xi^{-1})\sigma(\xi D\xi^{-1}) = \xi G\xi^{-1}\cap (D\sigma D) = \xi(\Gamma\beta\Gamma)\xi^{-1}$ for some $\beta\in G\cap Z$. Supposing that $\Gamma\alpha$ are the single cosets in $\Gamma\beta\Gamma$, then we see that $(\xi\Gamma\xi^{-1})(\xi\alpha\xi^{-1})$ are the single cosets of $G\cap(\xi D\xi^{-1})\sigma(\xi D\xi^{-1})$. Note that $\xi\in \mfrak{D}[2, 2]$ and that
\begin{align*}
	(f||_k\xi^{-1})|A_q &= \sum_{\alpha} j^k(\xi^{-1}, \xi\alpha\xi^{-1}z)^{-1}
J^k(\xi\alpha\xi^{-1}, z)^{-1}  f(\alpha\xi^{-1}z), \\
	(f|A_q)||_k\xi^{-1} &= j^{k}(\xi^{-1}, z)^{-1}\sum_{\alpha} 
J^k(\alpha, \xi^{-1}z)^{-1} f(\alpha\xi^{-1}z). 
\end{align*} 
So all that remains is to show that 
\[ 
	j^{k}(\xi^{-1}, \xi\alpha\xi^{-1}z)^{-1} J^k(\xi\alpha\xi^{-1}, z)^{-1} 
= j^{k}(\xi^{-1}, z)^{-1}J^k(\alpha, \xi^{-1}z)^{-1}, 
\] 
and both sides of this equation are indeed equal to $J^k(\alpha\xi^{-1}, z)^{-1}$ by various properties of the factors of automorphy involved 
(\cite[(1.9c), (2.1a)]{Shimurahalf} and the usual cocycle relation).
\end{proof}

\begin{proposition} \label{eigenspan} 
The space $\mathcal{M}_k(\Gamma, \psi)$ has a basis consisting of eigenforms 
for the space $\mathcal{R}_0^{\mfrak{c}}$. 
\end{proposition}

\begin{proof} This is adapted from \cite{Andrianovearly}, Theorem 1.3.4. By \cite[Lemma 4.5]{Shimurahalf} we 
have that $T_{q, \psi}$ (or $A_q$) is Hermitian on cusp forms provided 
$q_p\in \Circle_p$ for $p\mid\mfrak{c}$. From this it follows immediately 
that $\mathcal{S}_k^n(\Gamma, \psi)$ has a basis of eigenforms for 
$\mathcal{R}_0^{\mfrak{c}}$.

For the Eisenstein series part $\mathcal{E}_k^n(\Gamma, \psi)$ we use 
induction on $n$. By \cite[p. 210]{Koblitz} the space 
$\mathcal{M}_k^1(\Gamma, \psi)$ has a basis of eigenforms for 
$\mathcal{R}_0^{\mfrak{c}}$.

We make three claims, which hold for all $1\leq n\in\mathbb{Z}$. 
\begin{enumerate}[(1)]
	\item The space $\mathcal{E}_k^n(\Gamma, \psi)$ is invariant under $(\mathcal{R}_0^n)^{\mfrak{c}}$.
	\item There exists an epimorphism $(\mathcal{R}_0^n)^{\mfrak{c}}\to(\mathcal{R}_0^{n-1})^{\mfrak{c}}; A\mapsto A^*$ such that $\Phi_{\xi}(f|A) = (\Phi_{\xi})|A^*$ for all $\xi\in X_{n-1}$, $f\in\mathcal{M}_k^n(\Gamma)$.
	\item The space $\Phi_{\xi}\mathcal{E}_k^n(\Gamma, \psi)$ is invariant under $(\mathcal{R}_0^{n-1})^{\mfrak{c}}$ for all $\xi\in X_{n-1}$.
\end{enumerate}
Claim \textbf{(1)} follows from the self-adjointness 
cited in the first paragraph combined with the fact that cusp forms are 
readily seen to be preserved. Claim \textbf{(2)} is given by the local maps 
$\Psi(\cdot, \psi_p(p)^{-1}p^{n-[k]})$  combined with Lemma \ref{heckeslash} above.
Claim \textbf{(3)} follows from the previous 
two claims; indeed let $A = A_0^*$ for $A\in (\mathcal{R}_0^{n-1})^{\mfrak{c}}$ and 
$A_0\in (\mathcal{R}_{0}^n)^{\mfrak{c}}$, then
\[ 
	(\Phi_{\xi}\mathcal{E}_k^n(\Gamma, \psi))|A = \Phi_{\xi}(\mathcal{E}_k^n(\Gamma, \psi)|A_0)\subseteq\Phi_{\xi}\mathcal{E}_k^n(\Gamma, \psi). 
\]

So assume the proposition holds for $n-1$. By the induction hypothesis 
we obtain, for each $\xi\in X_{n-1}$, a basis of eigenforms for $\mathcal{M}_k^{n-1}
(\xi\Gamma\xi^{-1}\cap\mfrak{P}^{n, n-1}, \psi)$. Call this $\mathcal{B}_{\xi}$. Since the subspace $\Phi_{\xi}\mathcal{E}_k^n
(\Gamma, \psi)\subseteq\mathcal{M}_k^{n-1}(\xi\Gamma\xi^{-1}
\cap\mfrak{P}^{n, n-1}, \psi)$ is invariant under 
($\mathcal{R}_0^{n-1})^{\mfrak{c}}$ we can obtain (from $\mathcal{B}_{\xi}$) 
a basis $\mathcal{C}_{\xi}$ for $\Phi_{\xi}\mathcal{E}_k^n(\Gamma, \psi)$ 
consisting of eigenforms of $(\mathcal{R}_{0}^{n-1})^{\mfrak{c}}$. Let $\mathcal{C}_X$ 
denote the resultant product basis of $\Phi_{\star}\mathcal{E}_k^n(\Gamma, \psi)$. 
As $\ker(\Phi_{\star}) = \mathcal{S}_k^n(\Gamma, \psi)$ we have that $\Phi_{\star}$ 
is injective on $\mathcal{E}_k^n(\Gamma, \psi)$ so that the inverse image of 
$\mathcal{C}_X$, call it $\mathcal{C}_{\star}$, gives a basis for 
$\mathcal{E}_k^n(\Gamma, \psi)$.

Finally if $E\in \mathcal{C}_{\star}$ then we claim that it is an 
eigenform for $(\mathcal{R}_{0}^{n})^{\mfrak{c}}$. By definition $\Phi_{\star} E$ 
is 0 at all places except one $\xi_0$, at which it is some eigenform, say 
$F_{\xi_0}$ with eigenvalues $\Lambda$. Then if $A_q\in(\mathcal{R}_{0}^{n})^{\mfrak{c}}$ 
we have
\[ 
	\Phi_{\xi_0}(E|A_q) = (\Phi_{\xi_0}E)|A_q^* = F_{\xi_0}|A_q^* 
= \Lambda(A_q^*)F_{\xi_0},
\]
so that $\Phi_{\star}(E|A_q) = \Lambda(A_q^*)\Phi_{\star}(E)$. 
By injectivity of $\Phi_{\star}$ we are done.
\end{proof}

\begin{proposition} \label{eigenspantwo} 
Let $V\subseteq\mathcal{M}_k(\Gamma, \psi)$ be an eigenspace for $\mathcal{R}_0^{\mfrak{c}}$ with eigenvalues 
given by $\Lambda$, then it is spanned by 
$V\cap\mathcal{M}_k^n(\mathbb{Q}(\psi_{\mfrak{c}}, \Lambda))$.
\end{proposition}

\begin{proof} Write
\[ 
	V = \mathcal{M}_k^n(\Gamma, \psi, \Lambda) := 
\{f\in\mathcal{M}_k^n(\Gamma, \psi)\mid f|A = \Lambda(A)f\ 
\text{for all $A\in\mathcal{R}_0^{\mfrak{c}}$}\}.
\]
By the previous proposition the space $\mathcal{M}_k^n(\Gamma, \psi)$ is spanned by eigenforms for $\mathcal{R}_0^{\mfrak{c}}$, and by Lemma 5.1 of \cite{Bouganis} the action of $\mathcal{R}_0^{\mfrak{c}}$ preserves $\mathcal{M}_k^n(\Gamma, \psi, \mathbb{Q}(\psi, \Lambda))$. As we have a ring of $\mathbb{Q}(\psi, \Lambda)$-linear transformations on $\mathcal{M}_k^n(\Gamma, \psi, \mathbb{Q}(\psi, \Lambda)$ the argument of \cite[p. 233]{Shimurabook} follows.
\end{proof}

So, if $k>2n$, we obtain an equality of two different direct sums for 
$\mathcal{M}_k^n(\Gamma, \psi)$. From Proposition \ref{eigenspan} 
one of these consists of eigenspaces for $\mathcal{R}_0^{\mfrak{c}}$
and by Theorem \ref{decomp} the other one consists of 
$\mathcal{E}_k^{n, r}(\Gamma, \psi)$. 
By Theorem \ref{harris} we have that $\mathcal{E}_k^{n, r}(\Gamma, \psi)$ 
contains entire eigenspaces for $\mathcal{R}_0^{\mfrak{c}}$. So by the basic 
properties of direct sums we see that each $\mathcal{E}_k^{n, r}(\Gamma, \psi)$
 is itself a direct sum of eigenspaces.

For any character $\psi$ let $\Lambda_{k, \psi} = \Lambda_{k, \psi}^n\subseteq\Hom(\mathcal{R}_0^{\mfrak{c}}, \mathbb{C})$ 
be the finite subset such that
\begin{align}
	\mathcal{M}_k^n(\Gamma, \psi) = \bigoplus_{\Lambda\in\Lambda_{k, \psi}}
\mathcal{M}_k^n(\Gamma, \psi, \Lambda); \label{eigendecompL}
\end{align}
let $\mathbb{Q}(\Lambda_{k, \psi})/\mathbb{Q}$ be the field 
generated by all the values of $\Lambda$, for all $\Lambda\in \Lambda_{k, \psi}$.
The above discussion in conjunction with Proposition \ref{eigenspantwo} 
gives the following result for $0\leq r\leq n-1$ (the cusp form case $r=n$ is
already known).

\begin{corollary}\label{harriscoro} 
Let $0\leq r\leq n$ be integers and assume that $k>2n$. 
Then $\mathcal{E}_k^{n, r}(\Gamma, \psi)$ is spanned by 
$\mathcal{E}_k^{n, r}(\Gamma, \psi, \mathbb{Q}(\psi_{\mfrak{c}}, \Lambda_{k, \psi}))
:= \mathcal{E}_k^{n, r}(\Gamma, \psi)\cap \mathcal{M}_k^n(\mathbb{Q}(\psi_{\mfrak{c}}, \Lambda_{k, \psi}))$.
\end{corollary}

We need such an algebraic basis at other cusps as well. Let $\zeta_{\mfrak{m}} := e^{2\pi i\frac{1}{N(\mfrak{m})}}$ denote the $N(\mfrak{m})$th root of unity for an integral ideal $\mfrak{m}$ and recall $\zeta:\mfrak{M}\to\mathbb{T}$ as the character, see property (\ref{h1}), such that $h(\sigma, z)^2 = \zeta(\sigma)j(\pr(\sigma), z)$. Let $\zeta_{\star} = \zeta|_X$ where $X = \bigcup_r X_r$.

\begin{theorem}\label{nameless} 
Let $K/\mathbb{Q}$ be an algebraic field extension and let $f\in\mathcal{M}_k^n(\Gamma, K)$. Then for all $0\leq r\leq n$ and all $\xi\in X_r$ we have
 $f||_k\xi^{-1}\in\mathcal{M}_k^n(\xi\Gamma\xi^{-1}, K(\zeta_{\mfrak{c}}, \zeta_{\star}))$.
\end{theorem}

\begin{proof} 
In the integral weight case -- $\ell\in\mathbb{Z}$ and $g\in \mathcal{M}_{\ell}(\Gamma, K)$ --
Proposition 1.8 in \cite[p. 146]{Faltings} gives
$g||_{\ell}\xi^{-1}\in\mathcal{M}_{\ell}(\xi\Gamma\xi^{-1}, K(\zeta_{\mfrak{c}}))$. The half-integral case can be deduced from this 
via the use of the theta series $\theta(z) := \sum_{a\in \mathbb{Z}^n}
e(\frac{a^Tz a}{2})$. This belongs to $\mathcal{M}_{\frac{1}{2}}(\mathbb{Q})$ and, by the second equation of Proposition 1.3 in \cite{Shimuratheta}, we have that $\theta||_{\frac{1}{2}}\xi^{-1} = \theta$ has rational coefficients for each $\xi$.

Take $f$ as stated in the theorem, then $\theta f$ has integral weight 
$k+\frac{1}{2} = [k+1]$ and it has coefficients in $K$. 
So $(\theta f)||_{[k+1]}\xi^{-1}$ 
has coefficients in $K(\zeta_{\mfrak{c}})$ for any $\xi\in X_r$. 
We get
\begin{align*}
	((\theta f)||_{[k+1]}\xi^{-1})(z) 
&= j(\xi^{-1}, z)^{-[k+1]} \theta(\xi^{-1}z) f(\xi^{-1}z) \\
	&=j(\xi^{-1}, z)^{-1} h(\xi^{-1}, z)^2
(\theta||_{\frac{1}{2}}\xi^{-1})(z)(f||_k\xi^{-1})(z) \\
	&=\zeta(\xi^{-1}) \theta(z)(f||_k\xi^{-1})(z) 
\end{align*}
using property (\ref{h1}) of the factor $h(\sigma, z)$ in the last line, and 
the definition of the slash operator for half-integral $k$ along the way. 
Considering $\theta$ as an element of $\mathbb{Q}[[q]]$ with 
$q=e^{2\pi i}$ then it is an invertible power series since it has a
non-zero constant coefficient. So considering $\theta^{-1}
\in\mathbb{Q}[[q]]$ we have
\[ 
	f||_k\xi^{-1} = \zeta(\xi)\theta^{-1}
(\theta f)||_{[k+1]}\xi^{-1}\in \mathcal{M}_k(\xi\Gamma\xi^{-1}, K(\zeta_{\mfrak{c}}, \zeta_{\star})).
\]
\end{proof}

\begin{remark} For certain congruence subgroups one can remove the $\zeta_{\star}$. For example, if $\Gamma$ has cusps only at $0$ and $\infty$ then $X = \{I_{2n}, \iota\}$. In this case $\zeta(\iota) = (-i)^n$ by Proposition 1.1R of \cite{Shimuratheta} and we see that $\mathbb{Q}(\zeta_{\star})\subseteq\mathbb{Q}(\zeta_{\mfrak{c}})$ since $4\mid\mfrak{c}$. In general, however, it seems to be a necessary addition -- something which can be seen by Proposition 1.4 of \cite{Shimuratheta} and the subsequent paragraph detailing this proposition's non-triviality in contrast to the integral-weight case.
\end{remark}

\begin{corollary} \label{cuspbasis} 
Let $0\leq r\leq n$ be integers and assume $k>2n$. Then 
$\mathcal{E}_k^{n, r}(\xi\Gamma\xi^{-1}, \psi)$ is spanned by 
$\mathcal{E}_k^{n, r}(\xi\Gamma\xi^{-1}, \psi, \mathbb{Q}(\Lambda_{k, \psi}, G(\psi), \zeta_{\star}))$.
\end{corollary}

\begin{proof} 
As $\mathcal{E}_k^{n, r}(\xi\Gamma\xi^{-1}, \psi) = 
\mathcal{E}_k^{n, r}(\Gamma, \psi)||_k\xi^{-1}$ this follows from 
Corollary \ref{harriscoro} and Theorem \ref{nameless} above.
\end{proof}

The previously defined map $\Phi_{\star}$ provides a useful isomorphism
from which we can determine the rationality of $\Phi_{\star}f$ given that of $f$.

\begin{theorem}[(\cite{Shimuraeis}, p. 582; \cite{Shimuraexp}, p. 347)]\label{isom} 
Let $k>2n$ and fix $r<n$. Then 
\begin{align*}
	\Phi^{n-r}_{\star}:\mathcal{E}_k^{n, r}(\Gamma, \psi) &\to 
\prod_{\xi\in X_r}\mathcal{S}_k^r(\xi\Gamma\xi^{-1}\cap
\mfrak{P}^{n, r}, \psi) 
\end{align*}
is a $\mathbb{C}$-linear isomorphism.
\end{theorem}

\begin{corollary}\label{phirational} 
If $f\in\mathcal{E}_k^{n, r}(\Gamma, \psi)$ with $k>2n$, then 
$f\in\mathcal{M}_k(\Gamma, \mathbb{Q}(\Lambda_{k, \psi}, G(\psi), \zeta_{\star}))$ if and only if
\[ 
	\Phi_{\star}^{n-r}f\in\prod_{\xi\in X_r} 
\mathcal{S}_k^r(\xi\Gamma\xi^{-1}\cap\mfrak{P}^{n, r}, \psi, \mathbb{Q}(\Lambda_{k, \psi}, G(\psi), \zeta_{\star})). 
\]
\end{corollary}

\begin{proof} Theorem \ref{nameless} and the fact that 
$\Phi^{n-r}(\mathcal{M}_k^n(\Gamma, \psi, L))\subseteq 
\mathcal{M}_k^r(\Gamma, \psi, L)$ for any subfield $L\subseteq \mathbb{C}$ 
gives necessity.

For sufficiency, let $\{g_1^n, \dots, g_m^n\}$ be a basis of $\mathbb{Q}(\Lambda_{k, \psi}, G(\psi), \zeta_{\star})$-rational forms for $\mathcal{E}_k^{n, r}(\Gamma, \psi)$. By Corollary \ref{cuspbasis} 
there also exists a basis of $\mathbb{Q}(\Lambda_{k, \psi}, G(\psi), \zeta_{\star})$-rational forms for each of the spaces $\mathcal{S}_k^r(\xi\Gamma\xi^{-1}
\cap\mfrak{P}^{n, r}, \psi)$; let $\{g_1^r, \dots, g_m^r\}$ be the product 
basis for $\prod_{\xi}\mathcal{S}_k^r(\xi\Gamma\xi^{-1}\cap\mfrak{P}^{n, r}, \psi)$ 
obtained out of this. Thus each $g_i^r$ is $0$ for all of $X_r$ except for one $\xi$ 
whereby it is some $\mathbb{Q}(\Lambda_{k, \psi}, G(\psi), \zeta_{\star})$-rational element of $\mathcal{S}_k^r(\xi\Gamma\xi^{-1}\cap\mfrak{P}^{n, r}, \psi)$. Assume further that it is ordered so that 
$\Phi_{\star}^{n-r}(g_i^n) = g_i^r$ for all $i$. Writing $f = \sum_{i=1}^m 
\alpha_i g_i^n$ then we claim that $\alpha_i\in\mathbb{Q}(\Lambda_{k, \psi}, G(\psi), \zeta_{\star})$. 
By assumption 
\[ 
	\Phi_{\star}^{n-r} f = \sum_{i=1}^m \alpha_i g_i^r \in \prod_{\xi\in X_r}
\mathcal{S}_k^r(\xi\Gamma\xi^{-1}\cap\mfrak{P}^{n, r}, \psi,  \mathbb{Q}(\Lambda_{k, \psi}, G(\psi), \zeta_{\star})). 
\]
If $m=1$ and assuming $f\neq 0$ then by Lemma 8.2 (2) and Lemma 8.11 (4) 
in \cite{Shimuraeis} there exists $\xi\in X_r$ whereby $\Phi^{n-r}(f||_k\xi^{-1})\neq 0$ 
and thus $\Phi^{n-r}(g_1^n||_k\xi^{-1})\neq 0$. Then as $\alpha_1 = (g_1^r)^{-1}
\Phi_{\star}^{n-r}f$ we immediately see that $\alpha_1\in\mathbb{Q}(\Lambda_{k, \psi}, G(\psi), \zeta_{\star})$. 
The rest follows by induction on $m$.
\end{proof}

All of the above results allow us to now prove a particular case of Garrett's conjecture in Theorem \ref{harrismain} below.

\begin{lemma}[(Shimura, \cite{Shimuraeis}, p.578)]\label{phieis} 
If $f\in\mathcal{S}_k^r(\xi\Gamma\xi^{-1}\cap\mfrak{P}^{n, r}, \psi)$, $\xi\in X_r$, and $k>n+r+1$ then we have
\begin{align*}
	\Phi^{n-r}[E_k^{n, r}(z; f, \psi, \xi\Gamma\xi^{-1})||_k\xi\nu^{-1}] = 
\begin{cases} 
	f & \text{if $\nu = \xi$},\\ 
	0 & \text{if $\nu\in X_r$ and $\nu\neq \xi$}. 
\end{cases}
\end{align*}
\end{lemma}

\begin{remark} The above lemma is given in \cite{Shimuraeis} with trivial character and the proof then follows directly from Lemma 8.5 of that paper. This lemma (8.5) clearly applies for non-trivial character Klingen Eisenstein series, hence the above formulation.
\end{remark}

For any $f\in\mathcal{E}_k^{n, r}(\Gamma, \psi)$ and any $0\leq r\leq n$ define
\[ 
	F_k^{n, r}(z; f, \psi, \Gamma):= \sum_{\xi\in X_r} 
E_k^{n, r}(z;\Phi_{\xi}^{n-r}f, \psi, \xi\Gamma\xi^{-1})||_k\xi\in\mathcal{E}_k^{n, r}(\Gamma, \psi). 
\]

\begin{theorem} \label{harrismain} Let $0\leq r\leq n$ and $f\in\mathcal{E}_k^{n, r}
(\Gamma, \psi, \mathbb{Q}(\Lambda_{k, \psi}, G(\psi), \zeta_{\star}))$ with $k>n+r+1$. Then
$F_k^{n, r}(z; f, \psi, \Gamma)\in\mathcal{E}_k^{n, r}(\Gamma, \mathbb{Q}(\Lambda_{k, \psi}, G(\psi), \zeta_{\star}))$. 
\end{theorem}

\begin{proof} For any $\nu\in X_r$ we have 
$\Phi_{\eta}^{n-r}F_k^{n, r}(z; f, \psi, \Gamma) = \Phi_{\nu}^{n-r}f$ by Lemma \ref{phieis}. By Theorem \ref{nameless}
$\Phi_{\nu}^{n-r}f$ has coefficients in $\mathbb{Q}(\Lambda_{k, \psi}, G(\psi), \zeta_{\star})$ for each 
$\nu\in X_r$. If $0\leq r\leq n-1$ then by Corollary \ref{phirational} $F_k^{n, r}(z; f, \psi, \Gamma)$ also has coefficients in 
$\mathbb{Q}(\Lambda_{k, \psi}, G(\psi), \zeta_{\star})$. If $r=n$ then this is given immediately 
by Theorem \ref{nameless}.
\end{proof}

Let $\mathcal{E}_k^n := \prod_{r=0}^{n-1}\mathcal{E}_k^{n, r}$ denote the space of all Eisenstein series. The well-known decomposition $\mathcal{M}_k^n(\Gamma, \psi) = \mathcal{S}_k^n(\Gamma, \psi)\oplus\mathcal{E}_k^n(\Gamma, \psi)$ of Theorem \ref{decomp} is proven inductively and each step involves the use of the Eisenstein series $F_k^{n, r}$. Observing this proof along with Theorem \ref{harrismain} gives the following corollary. 

\begin{corollary} \label{maincoro} Assume that $k>2n$. Then, defining
\[
\mathscr{L} := \mathbb{Q}(\Lambda_{k, \psi}, G(\psi), \zeta_{\star}),
\]
we have
\[ 
	\mathcal{M}_k^n(\Gamma,\psi,  \mathscr{L}) 
= \mathcal{S}_k^n(\Gamma, \psi, \mathscr{L})\oplus 
\mathcal{E}_k^n(\Gamma, \psi, \mathscr{L}).
\]
\end{corollary}

\begin{proof} That this follows easily from Theorem \ref{harrismain} above will be clear after outlining the proof, taken from \cite[pp. 581 -- 582]{Shimuraeis} of the decomposition of Theorem \ref{decomp}. Let $f\in\mathcal{M}_k(\Gamma, \psi)$ be such that $\lambda_r(\Gamma\cap\mfrak{P}^{n, r})^{[k]} = 1$ for all $r$. 
Put $f_0 = F_k^{n, 0}(z; f, \psi, \Gamma)$. Then $\Phi_{\star}^n(f-f_0) = 0$ by Lemma \ref{phieis} so that $\Phi^{n-1}_{\nu}(f-f_0)$ is a cusp form for each $\nu$. Then put $f_1 
:= F_k^{n, 1}(z; f-f_0, \psi, \Gamma)$ and repeat the above procedure to get $f_2 = F_k^{n, 2}
(z;f-f_0-f_1, \psi, \Gamma)$ and so on. At the final step we obtain $0 = \Phi^0(f-f_0-f_1-\cdots -f_n) = f-f_0-\cdots -f_n$ and this gives Theorem \ref{decomp}. So we see that if $f$ has coefficients in 
$\mathbb{Q}(\Lambda_{k, \psi}, G(\psi), \zeta_{\star})$ then, by Theorem \ref{main}, so do each of
$f_0, f_1, \dots, f_n$. 
\end{proof}

\section{Special values}
\label{specialvalues}

The results of the previous section allow more special values to be determined via the method used earlier. This is done by relaxing growth conditions on holomorphic projection. We make use of the notation $z^{-k-|2s|} := z^{-k}|z|^{-2s}$.

\begin{definition} \label{moderategrowth} 
If $F\in C_k^{\infty}(\Gamma, \psi)$ then we say that $F$ is of 
\emph{moderate growth} if, for all $z\in\mathbb{H}_n$ and sufficiently large $\Re(s)\gg 0$, 
we have that the integral
\[ 
	\int_{\mathbb{H}_n} f(w) |\bar{w}-z|^{-k-|2s|}\Delta(w)^{k+s}d^{\times}w
\]
is absolutely convergent and admits an analytic continuation over $s$ to the point 
$s=0$.
\end{definition}

Forms of moderate growth are sent by the projection of Theorem \ref{holoproj} to $\mathcal{M}_k$ instead of $\mathcal{S}_k$.

\begin{theorem}\label{newholoproj} Assume $F\in C_k^{\infty}(\Gamma, \psi)$ has moderate growth and that $k>2n$. Assume that $F$ is of moderate growth. Then with $\mu(k, n)$, $c(\tau)$, and the projection map defined as in Theorem \ref{holoproj}, we have $\mathbf{Pr}(F)\in \mathcal{M}_k(\Gamma, \psi)$. Furthermore $\langle F, g\rangle = \langle \mathbf{Pr}(F), g\rangle$ for any $g\in\mathcal{S}_k(\Gamma', \psi)$ and $\Gamma'\leq\Gamma$ of finite index.
\end{theorem}

The proof of this is also given by the study of a certain two-variable Poincar\'{e} series which is defined for variables $z, w\in\mathbb{H}_n$ and $s\in\mathbb{C}$ as 
\[
	P(z, w, s) := (\Delta(z)\Delta(w))^s\sum_{\gamma\in \Gamma}\psi_{\mfrak{c}}^{-1}(|a_{\gamma}|)j_{\gamma}^k(z)^{-1}|j(\gamma, z)|^{-2s}|\gamma z+w|^{-k-|2s|}.
\]
This converges absolutely and uniformly on products $V(d)\times V(d)$ for $\Re(2s)>2m-k+1$, $d>0$, and $V(d) := \{z\in\mathbb{H}_n\mid y>dI_n, \tr(x^Tx)\leq d^{-1}\}$, see \cite[p. 72]{Panchishkin}. This series has been altered from the definition of the integral-weight version found in \cite{Panchishkin} only by the change in the factor of automorphy $j_{\gamma}^k(z)^{-1}$. Once it is shown that this series exhibits the analogous three properties to (4.9), (4.10), and (4.12) of \cite[p. 72]{Panchishkin}, then the proof of Theorem \ref{newholoproj} follows precisely as is found there.

\begin{proposition} \label{symmetry} For any $\gamma\in\Gamma$ let $\gamma' := \begin{psmallmatrix} I_n & 0 \\ 0 & -I_n\end{psmallmatrix}\gamma^{-1}\begin{psmallmatrix} I_n & 0 \\ 0 & -I_n\end{psmallmatrix}$. Then
\[
	j_{\gamma}^k(z)|\gamma z+ w|^k = j_{\gamma'}^k(w)|\gamma'w+z|^k.
\]
\end{proposition}

\begin{proof} By routine calculation $\big||\mu(\gamma, z)(\gamma z+ w)|^{\kappa}\big| = \big||\mu(\gamma', w)(\gamma' w+z)|^{\kappa}\big|$ for any $\kappa\in\frac{1}{2}\mathbb{Z}$. We claim that
\[
	\frac{h_{\gamma'}(w)}{|h_{\gamma'}(w)|} = \frac{h_{\gamma}(z)}{|h_{\gamma}(z)|}\in\mathbb{T}.
\]
These are constants independent of $w$ and $z$ respectively. Proposition 2.6 of \cite{Shimuratheta} gives that
\[
	\frac{h((\gamma')^{-1}, w)}{|h((\gamma')^{-1}, w)|} = \overline{\frac{h_{\gamma}(z)}{|h_{\gamma}(z)|}},
\]
since $\gamma' = (\gamma^*)^{-1}$ in Shimura's notation. We have $h_{\gamma'}(z) = h((\gamma')^{-1}, z)^{-1}$ using the cocycle relation of (\ref{h3}), which gives the claim and so
\begin{align*}
	h_{\gamma}(z)|\gamma z + w|^{\frac{1}{2}} &= \frac{h_{\gamma}(z)}{|h_{\gamma}(z)|}\big||\mu(\gamma, z)(\gamma z+w)|^{\frac{1}{2}}\big| \\
& = \frac{h_{\gamma'}(w)}{|h_{\gamma'}(w)|}\big||\mu(\gamma', w)(\gamma'w+z)|^{\frac{1}{2}}\big| \\
& = h_{\gamma'}(z)|\gamma' w +z|^{\frac{1}{2}},
\end{align*}
which gives the proposition.
\end{proof}

The above proposition proves the first two of the following three properties:
\begin{align}
	P(z, w, s) &= P(w, z, s), \label{P1} \\
	P(\gamma_1z, \gamma_2w, s) &=\psi_{\mfrak{c}}(|a_{\gamma_1}a_{\gamma_2}|)j_{\gamma_1}^k(z) j_{\gamma_2}^k(w)P(z, w, s), \label{P2} \\
	\langle F(w), P(-\bar{z}, w, s)\rangle &= \mu F(z), \label{P3}
\end{align}
for any $F\in C_k^{\infty}(\Gamma, \psi)$ such that the integral of (\ref{P3}) converges, and for some constant $\mu$ given in \cite[p. 73]{Panchishkin}. By definition the left-hand side of (\ref{P3}) is
\[
	(-1)^{ns}\Delta(w)^s\int_{\Gamma\bslsh\mathbb{H}_n}\sum_{\gamma\in\Gamma}\psi_{\mfrak{c}}(|a_{\gamma}|)F(z)\overline{j_{\gamma}^k(z)^{-1}}|j(\gamma, z)|^{-2s}|\gamma\bar{z}-w|^{-k-|2s|}\Delta(z)^{k+s}d^{\times}z.\]
Now use that $\psi_{\mfrak{c}}(|a_{\gamma}|)F(z) = j_{\gamma}^k(z)^{-1}F(\gamma z)$ and $\Delta(z) = |j(\gamma, z)|^2\Delta(\gamma z)$ to get
\begin{align*}
	(-1)^{ns}&\Delta(w)^s\int_{\Gamma\bslsh\mathbb{H}_n}\sum_{\gamma\in\Gamma}F(\gamma w)|\gamma\bar{w}-z|^{-k-|2s|}\Delta(\gamma z)^{k+s}d^{\times}z \\
	&=(-1)^{ns}\Delta(w)^s\int_{\mathbb{H}_n}F(w)|\bar{w}-z|^{-k-|2s|}\Delta(z)^{k+s}d^{\times}z,
\end{align*}
which is exactly of the form found in (4.14) of \cite[p. 73]{Panchishkin}. So the rest of that proof using Cayley transforms applies, and we get property (\ref{P3}). Note that the above integral is convergent and has analytic continuation to $s = 0$ precisely when $F$ is of moderate growth. 

To finish the proof of the projection in this case, set $K(z, w, s) := \mu^{-1}P(-\bar{z}, w, s)$ and then define $\mathbf{Pr}(F)(z) := \langle F(w), K(z, w, s)\rangle|_{s=0}$. The reader is referred to \cite[pp. 74--75]{Panchishkin} for the details here.

\begin{proposition}\label{newbounds} 
Let $k$ be a half-integral weight, $\ell\in\frac{1}{2}\mathbb{Z}$, and $b > \frac{n+1}{2}$. If $g\in\mathcal{M}_{\ell}(\Gamma, \psi)$ then $F^*(z) := g(z)H_{k-\ell}(z, b)$ is of moderate growth provided 
\[
	\tfrac{k+\ell}{2} - n(k+1) - 2 < b < n(k+1) - \tfrac{k+\ell}{2} + 2.
\]
\end{proposition}

\begin{proof} Set $s=0$ in the integral characterising moderate growth in Definition \ref{moderategrowth}. 
Fixing $z\in\mathbb{H}_n$, then let $w=x+iy$ with $\lambda_j$ being the 
eigenvalues of $y$. Notice that $|\bar{w}-z|$ is a polynomial in $x_{ij}, y_{ij}$ 
of degree $n>0$ which is $|-iy-z|$ as $x\to 0$. Hence $\|\bar{w}-z\|^{-k}$ decays as $|x|\to\pm\infty$ and is finite as $x\to 0$.
Then, by Corollary \ref{boundsonemod}, we may write for some constant $\nu$
\[ 
	\int_{\mathbb{H}_n} |F^*(w)|\|\bar{w}-z\|^{-k}|y|^{k-n-1}dydx
\leq \nu\int_Y\upsilon(y)|P(y)|^{-k} dy,
\]
where $P(y)$ is a polynomial in $y_{ij}$ of degree $n$, and 
\[
	\upsilon(y) := \prod_{j=1}^n (1-\lambda_j^{-\ell})\left(\lambda_j^{b-\frac{k-\ell}{2}} + 
\lambda_j^{-b-\frac{k-\ell}{2}}\right)\lambda_j^{k-1-n}.
\]
Let $\widetilde{\Lambda} := \{\diag[\lambda_1, \dots, \lambda_n]\mid 0<\lambda_1\leq \cdots\leq \lambda_n\}$. As is done in the proof of Corollary 2 of 
\cite{Sturm} we may make the 
substitution $y = U\Lambda U^T$, where $U\in O_n(\mathbb{R})$ and 
$\Lambda\in\tilde{\Lambda}$. If $\lambda_i$ are all distinct then this is unique up to multiplication of $U$ by $\diag[\pm 1, \dots, \pm 1]$. Now $\upsilon(y)$ and the Jacobian of the transformation are independent of $U$, therefore the integral over $O_n(\mathbb{R})$ will be finite so long as the integral over $\widetilde{\Lambda}$ is, so it is enough to show that
\[ 
	\int_{\tilde{\Lambda}}\upsilon(y)|P(\Lambda)|^{-k} |J(\lambda_1, \dots, \lambda_n)
|d\lambda_1\cdots d\lambda_n<\infty,
\]
where $J(\lambda_1, \dots, \lambda_n)$ is the Jacobian matrix. To do this we check the limits $\lambda_j\to 0$ and $\lambda_j\to\infty$. 
Firstly, as $\lambda_j\to 0$, then $|P(\Lambda)|^{-k}\to\|z\|^{-k}$ is just finite so we require 
the exponent of each $\lambda_j$ to be greater than $-1$ (and $b>\frac{n+1}{2}$ 
in order for $H$ to be defined). This just gives us the original bounds 
found in Corollary \ref{boundstwo} for bounded growth. For the limit $\lambda_j\to\infty$ we 
have that $|P(\Lambda)|^{-k}$ decays to order $nk$, so as long as the exponent 
of $\lambda_j$ in $\upsilon(y)$ is $\leq nk$ we obtain convergence. That is
\begin{align*}
	b-\tfrac{k-\ell}{2}+k-1-n &\leq nk, \\
	-b-\tfrac{k-\ell}{2} + k-1-n&\leq nk,
\end{align*}
giving the bounds in the statement.
\end{proof}

If $g\in\mathcal{M}_{\ell}(\Gamma, \psi')$ for $\ell\in\frac{1}{2}\mathbb{Z}$, then define
\begin{align*}
	\Omega^+_g &:= \{m\in\tfrac{1}{2}\mathbb{Z}\mid \tfrac{n-2m+2k-2\ell}{4}\in\mathbb{Z}, n<m\leq k-\ell+\tfrac{n}{2}\}, \\
	\Omega^-_g &:= \{m\in\tfrac{1}{2}\mathbb{Z}\mid \tfrac{2m-3n+2k-2\ell-2}{4}\in\mathbb{Z}, 
\tfrac{3n}{2}+1-k+\ell\leq m\leq n\},
\end{align*}
and put $\Omega_g := \Omega^-_g\cup\Omega^+_g$. 

\begin{proposition}\label{newerintegralexp} 
Exclude case (\ref{X}). Let $\ell\in\frac{1}{2}\mathbb{Z}$ and $g\in\mathcal{M}_{\ell}(\Gamma, \psi')$. Assume that $k>2n$. In case (\ref{R1}) set $m_0 := \frac{k+\ell-3}{2}$ and in case (\ref{R2}) set $m_0 :=\frac{2k+2\ell+n-1}{4}-\frac{n+1}{2}$. For all other cases
\[
	m_0 := \begin{cases} \frac{2k+2\ell +2m-n}{4}-\frac{n+1}{2} &\text{if $m>n$}, \\
	\frac{2k+2\ell+3n-2m+2}{4}-\frac{n+1}{2} &\text{if $m\leq n$}. \end{cases}
\]
For every 
$m\in\Omega_g$ there exists $K_{\mathcal{S}}(m, g)\in\mathcal{S}_k(\Gamma, \psi)$,
whose Fourier coefficients belong to $\mathbb{Q}_{ab}(g, \Lambda_{k, \psi}, G(\psi), \zeta_{\star})$, such that
\begin{align*}
	\frac{(4\pi)^{nm_0}}{\pi^{\beta_m}\omega_{\ell}(m, \bar{\psi}\psi')}\Gamma_n(m_0)^{-1}&\left\langle f, g\mathcal{E}_{k-\ell}^{\bar{\psi}\psi'}\left(\cdot, \tfrac{2m-n}{4}\right)\right\rangle = \pi^{n(k-r)-\frac{3n^2+2n+\delta}{4}}\langle f, K_{\mathcal{S}}(m, g)\rangle
\end{align*}
for all $f\in\mathcal{S}_k(\Gamma, \psi)$. Moreover $K_{\mathcal{S}}(m, g)^{\sigma} = K_{\mathcal{S}}(m, g^{\sigma})$ for all $\sigma\in\Aut(\mathbb{C}/\mathbb{Q}(\Lambda_{k, \psi}, G(\psi), \zeta_{\star}))$.
\end{proposition}

\begin{proof} 
Much of this remains the same as the proof of Proposition \ref{newintegralexp}. To apply holomorphic projection we ensure moderate growth of $g(z)H_{k-\ell}(z, \frac{2m-n}{4})$ when $m\in\Omega_g^+$. The analytic continuation of the Eisenstein series is given by a functional equation in $s\mapsto \frac{n+1}{2}-s$, so for $m\in \Omega_g^-$ we consider the majorant $H_{k-\ell}(z, \frac{n+1}{2}-\frac{2m-n}{4})$. That $\frac{2m-n}{4}$ and $\frac{n+1}{2}-\frac{2m-n}{4}$ satisfy the bounds of Proposition \ref{newbounds} is immediate from the definition of $\Omega_g$. Moreover $\frac{2m-n}{4}\in\Omega_0$ -- allowing the application of Theorem \ref{eisenalg} -- and $k-r>n$ -- allowing the application of Lemma \ref{rationalintegral}. The changes in the definition of $m_0$ for the four separate cases is a result of the change to the order $r$ of the non-holomorphic Eisenstein series -- see Theorem \ref{eisenalg}.

Therefore in applying holomorphic projection, and replicating the proof of Proposition \ref{newintegralexp}, we obtain a holomorphic modular form $K(m, g)\in\mathcal{M}_k(\Gamma, \psi)$ with coefficients in $\mathbb{Q}_{ab}(g)$. By Corollary \ref{maincoro} this splits up as $K(m, g) = K_{\mathcal{S}}(m, g) + K_{\mathcal{E}}(m, g)$, where $K_{\mathcal{X}}(m, g)\in\mathcal{X}_k(\Gamma, \psi)$, for $\mathcal{X}\in\{\mathcal{S}, \mathcal{E}\}$, has coefficients in $\mathbb{Q}_{ab}(g, \Lambda_{k, \psi}, G(\psi), \zeta_{\star})$. Since $\langle f, K_{\mathcal{E}}(m, g)\rangle = 0$ we are done.
\end{proof}

Now set $\ell = \frac{n}{2}+\mu$ and assume $k> 2n$ in all cases. If $m\in \Omega^+_{\theta_{\chi}}$ and we are not in cases (\ref{R1}) or (\ref{R2}) then we obtain the same integral expression, (\ref{projintexp}), for $L_{\psi}(m, f, \eta)$. On the other hand, if (\ref{R1}), (\ref{R2}) hold, or $m\in\Omega^-_{\theta_{\chi}}$ then this will be slightly different since here the value $m_0$ required to apply Proposition \ref{newerintegralexp} above is no longer occuring naturally from the original expression in (\ref{intexp}). If $m\in\Omega^-_{\theta_{\chi}}$ then
\[
	m_0 = \tfrac{k+n+\mu-m}{2} = \left(\tfrac{m-n-1+k+\mu}{2}\right) + n - m +\tfrac{1}{2},
\]
from which
\begin{align*}
	(4\pi)^{n\left(\frac{m-n-1+k+\mu}{2}\right)} &= (4\pi)^{nm_0}(4\pi)^{\frac{n}{2}(2m-2n-1)}, \\
	\Gamma_n\left(\tfrac{m-n-1+k+\mu}{2}\right)\Gamma_n(m_0)^{-1}&\in\mathbb{Q}.
\end{align*}
Therefore, from (\ref{intexp}), we obtain
\begin{align*}
	L_{\psi}(m, f, \chi)\in &(2|2\tau|)^{-\frac{\delta}{2}}\pi^{\frac{n}{2}(2m-n)}c_f(\tau, 1)^{-1}\Lambda_{\mfrak{c}, \mfrak{y}}(m, \eta)\prod_{p\in\mathbf{b}}g_p\left((\psi^{\mfrak{c}}\chi^*)(p)p^{-m}\right) \\
	&\times (4\pi)^{nm_0}\Gamma_n(m_0)^{-1}\left\langle f, \theta_{\chi}\mathcal{E}(\cdot, \tfrac{2m-n}{4})\right\rangle\mathbb{Q},
\end{align*}
and, as before, multiplying both sides by $\pi^{-\beta_m}\omega_{\delta}(m, \bar{\eta})^{-1}$ and applying Proposition \ref{newerintegralexp} gives
\begin{align}
	\frac{(2|2\tau|)^{\frac{\delta}{2}}c_f(\tau, 1)L_{\psi}(m, f, \chi)}{\pi^{\beta_m+n(k+m-r)-\frac{5n^2+2n+\delta}{4}}\omega_{\delta}(m, \bar{\eta})} \in\Lambda_{\mfrak{c}, \mfrak{y}}(m, \eta)\prod_{p\in\mathbf{b}}g_p\left((\psi^{\mfrak{c}}\chi^*)(p)p^{-m}\right)\langle f, K_{\mathcal{S}}(m, \theta_{\chi})\rangle\mathbb{Q}. \label{newerprojintexp}
\end{align}
If we are in cases (\ref{R1}) or (\ref{R2}) then $(4\pi)^{n(\frac{m-n-1+k+\mu}{2})} = (4\pi)^{nm_0+n}$ and rationality of the $\Gamma$-factors is, again, preserved. Hence
\begin{align}
	\frac{(2|2\tau|)^{\frac{\delta}{2}}c_f(\tau, 1)L_{\psi}(m, f, \chi)}{\pi^{\beta_m+n(k-r)-\frac{n^2-4n+\delta}{4}}\omega_{\delta}(m, \bar{\eta})} \in\Lambda_{\mfrak{c}, \mfrak{y}}(m, \eta)\prod_{p\in\mathbf{b}}g_p\left((\psi^{\mfrak{c}}\chi^*)(p)p^{-m}\right)\langle f, K_{\mathcal{S}}(m, \theta_{\chi})\rangle\mathbb{Q}. \label{Rprojintexp}
\end{align}

We can also make some improvements on the bounds for $k$ in Theorem \ref{innerprod}. Let
\[
	c_m = \beta_m + n(k-r) -\tfrac{n^2-4n+\delta}{4}
\]
in cases (\ref{R1}) and (\ref{R2}). Otherwise let
\[
	c_ =\begin{cases} \beta_m + n(k-r) - \frac{n^2+\delta}{4}& \text{if $m>n$}, \\
	\beta_m + n(k+m-r) - \frac{5n^2+2n+\delta}{4} &\text{if $m\leq n$}. \end{cases}
\]
Recall $\eps\in\{0, 1\}$ as the value such that $\psi_{\infty}(x) = \sgn(x_{\infty})^{[k]+\eps}$.

\begin{theorem}\label{improvedinnerprod} Assume that $k>\max\{2n, \frac{3n}{2}+1+\eps\}$. If $f\in\mathcal{S}_k(\Gamma[\mfrak{b}^{-1}, \mfrak{bc}], \psi, \Lambda)$ is a Hecke eigenform for ideals $(\mfrak{b}^{-1}, \mfrak{bc})\subseteq 2\mathbb{Z}\times 2\mathbb{Z}$, a Hecke character $\psi$, and a system of eigenvalues $\Lambda$, then there exists a non-zero constant $\mu(\Lambda, k, \psi)$ -- dependent only on $\Lambda, k, \psi$ -- such that
\[
	\left(\frac{\langle f, g\rangle}{\mu(\Lambda, k, \psi)}\right)^{\sigma} = \frac{\langle f^{\sigma}, g^{\rho\sigma\rho}\rangle}{\mu(\Lambda_{\sigma}, k, \psi^{\sigma})},
\]
for any $g\in\mathcal{S}_k(\Gamma[(\mfrak{b}')^{-1}, \mfrak{b}'\mfrak{c}'], \psi)$, ideals $((\mfrak{b}')^{-1}, \mfrak{b}'\mfrak{c}')\subseteq \mfrak{b}^{-1}\times \mfrak{bc}$, and $\sigma\in\Aut(\mathbb{C}/\mathbb{Q})$.
\end{theorem}

\begin{proof} Take $\mfrak{m}$ as in the proof of Theorem \ref{innerprod}; since we have a larger set $\Omega(\theta_{\varphi_{\mfrak{m}}})$ of special values, we can change the special value of the $L$-function that defined our original constant $\mu'(\Lambda, k, \psi)$. Let
\[
	\mu(\Lambda, k, \psi) := 
			2^{\frac{\delta}{2}}\pi^{-c_{k-\eps}}i^{-\frac{n^2}{2}}\omega_{\delta}(k-\eps, \overline{\psi})^{-1}L_{\psi}(k-\eps, f).
\]
We have
\[
	\Omega^+_{\theta_{\mfrak{m}}} = \{m\in\tfrac{1}{2}\mathbb{Z}\mid \tfrac{k-m-\eps}{2}, n < m\leq k-\eps\}.
\]
Then since $k-n\in \Omega(\theta_{\varphi_{\mfrak{m}}})$ the rest of this proof follows exactly as that of Theorem \ref{innerprod}, but using the integral expression (\ref{newerprojintexp}) (resp. (\ref{Rprojintexp})) instead when $m\leq n$ (resp. cases (\ref{R1}) or (\ref{R2})). The bound $k > \frac{3n}{2} + 1 + \eps$ guarantees the non-vanishing of the $L$-function, and therefore of the constant, and the bound $k > 2n$ is there to ensure we can still apply holomorphic projection.
\end{proof}

\ref{main}
\begin{theorem}\label{newmain} 
Let $f\in\mathcal{S}_k(\Gamma, \psi, \Lambda)$ be an eigenform for a half-integral weight $k$, congruence subgroup $\Gamma = \Gamma[\mfrak{b}^{-1}, \mfrak{bc}]$ contained in $\mfrak{M}$, Hecke character $\psi$ satisfying (\ref{char1}) and (\ref{char2}), and system of eigenvalues $\Lambda$. Assume that $k >\max\{2n,  \frac{3n}{2}+1 + \eps\}$, where $\eps\in\{0, 1\}$ is given by $\psi_{\infty}(x) = \psi(x_{\infty})^{[k]+\eps}$. Let $\chi$ be a Hecke character, and choose $\mu\in\{0, 1\}$ so that we have $(\psi\chi)_{\infty}(x) = \sgn(x)^{[k]+\mu}$. Define the sets 
\begin{align*}
	\Omega_{n, k}^+ &:= \{m\in\tfrac{1}{2}\mathbb{Z}\mid \tfrac{m-k-\mu}{2}\in\mathbb{Z}, n<m\leq k-\mu\}, \\
	\Omega_{n, k}^- &:= \{m\in \tfrac{1}{2}\mathbb{Z}\mid \tfrac{m+k-\mu-1}{2}\in\mathbb{Z}, 2n+1-k+\mu\leq m\leq n\}, \\
	\Omega_{n, k} &:= \Omega_{n, k}^-\cup\Omega_{n, k}^+.
\end{align*}
Now if $\tau\in S_+$ is such that $c_f(\tau, 1)\neq 0$ and $m\in\Omega_{n, k}$ then define
\[ 
	Z_{\psi}(m, f, \chi) := |\tau|^{\frac{\delta}{2}}\pi^{-c_m}\mu(\Lambda, k, \psi)^{-1}\omega_{\delta}(m, \bar{\eta})^{-1}L_{\psi}(m, f, \chi). 
\]
We have $Z_{\psi}(m, f, \chi)^{\sigma} = Z_{\psi^{\sigma}}(m, f^{\sigma}, \chi^{\sigma})$ 
for any $\sigma\in\Aut(\mathbb{C}/\mathbb{Q}(\Lambda_{k, \psi}, G(\psi), \zeta_{\star}))$, where, recall, $\Lambda_{k, \psi}$ is defined by (\ref{eigendecompL}) and $\zeta_{\star}$ is the character of (\ref{h1}) restricted to all the cusps of $\Gamma\bslsh\mathbb{H}_n$.  Hence
\[
	Z_{\psi}(m, f, \chi)\in\mathbb{Q}(f, \chi, \Lambda_{k, \psi}, G(\psi), \zeta_{\star}).
\]
\end{theorem}

\begin{proof} Note that $\Omega^{\pm}_{\theta_{\chi}} = \Omega_{n, k}^{\pm}$. If $m\in\Omega_{n, k}^+$ and we are not in cases (\ref{R1}, \ref{R2}) then combine the integral expression of (\ref{projintexp}) with Proposition \ref{newerintegralexp} whereas, if $m\in\Omega_{n, k}^-$ (resp. (\ref{R1}, \ref{R2})), then use the integral expression of (\ref{newerprojintexp}) (resp. (\ref{Rprojintexp})) directly. This gives
\[
	Z_{\psi}(m, f, \chi) \in c_f(\tau, 1)^{-1}\Lambda_{\mfrak{c}, \mfrak{y}}(m, \eta)\prod_{p\in\mathbf{b}}g_p\left((\psi^{\mfrak{c}}\chi^*)(p)p^{-m}\right)\frac{\langle f, K_{\mathcal{S}}(m, \theta_{\chi})\rangle}{\mu(\Lambda, k, \psi)}\mathbb{Q},
\]
which is evidently $\sigma$-equivariant over $\Aut(\mathbb{C}/\mathbb{Q}(\Lambda_{k, \psi}, G(\psi), \zeta_{\star}))$ by Theorem \ref{improvedinnerprod}.
\end{proof}

\begin{remark} In all cases we actually have $c_m = n(k+m-n)$ (putting $\ell = \frac{n}{2} + \mu$ in the definitions of $\beta_m$ and $r$). They are therefore integers and agree with the powers of $\pi$ present in Theorem 28.8 of \cite{Shimurabook}. We present the powers of $\pi$ as a sum of its constituents in order to clarify the proofs of the main results throughout this paper.
\end{remark}

\begin{acknowledgements} I would like to thank my PhD supervisor Thanasis Bouganis 
for the direction and guidance of this paper, and for generally keeping my head 
screwed on through its intricacies. Funding was provided by Engineering and Physical Sciences Research Council
(Grant No. 000118421).
\end{acknowledgements}


\begin{thebibliography}{99}

\bibitem[And74]{Andrianovearly} 
 A.~Andrianov, 
`Euler Products Corresponding to Siegel Modular Forms of Genus 2', 
{\em Russian Mathematical Surveys}, \textbf{29} (3) (1974), 45--116.

\bibitem[And79]{Andrianov} 
 A.~Andrianov, 
`The Multiplicative Arithmetric of Siegel Modular Forms', 
{\em Russian Mathematical Surveys}, \textbf{34} (1) (1979), 75--148.

\bibitem[Boug18]{Bouganis} 
 T.~Bouganis, 
`On Special $L$-Values Attached to Metaplectic Modular Forms', 
{\em Mathematische Zeitschrift}, \textbf{3-4} (2018), 725--740.

\bibitem[FC80]{Faltings} 
 G.~Faltings \& C.-L.~Chai, 
{\em Degeneration of Abelian Varieties}, 
Springer-Verlag, Berlin Heidelberg (1980).

\bibitem[Garr84]{Garrett}
P.~B.~Garrett,
`Pullbacks of Eisenstein series: Applications',
{\em Progress in Mathematics}, \textbf{46} (1984), 114--137.

\bibitem[CG58]{Godement} 
 H.~Cartan, R.~Godement, {\em et al}, 
`Fonctions Automorphes', 
{\em Seminaire Cartan}, 10 (1958).

\bibitem[Harr81]{Harris} 
 M.~Harris, 
`The Rationality of Holomorphic Eisenstein Series', 
{\em Inventiones Mathematicae}, \textbf{63} (1981), 305--310.

\bibitem[Harr84]{Harristwo} 
 M.~Harris, 
`Eisenstein Series on Shimura Varieties', 
{\em Annals of Mathematics}, \textbf{119} (1984), 59--94.

\bibitem[Hay03]{hayashida}
S.~Hayashida, `Zeta function and Zharkovskaya's theorem on Siegel modular forms of half-integral weight', 
{\em Acta Arithmetica}, \textbf{108} (4) (2003), 391--399.

\bibitem[Kob84]{Koblitz} 
 N.~Koblitz, 
{\em Introduction to Elliptic Curves and Modular Forms}, 
Springer-Verlag, New York, (1984).

\bibitem[OOK89]{OOK}
Y.Y.~Oh, J.K.~Koo \& M.H.~Kim, `Hecke operators and the Siegel operator', 
{\em Journal of the Korean Mathematical Society}, \textbf{26} (2) (1989), 323--334.

\bibitem[Pan91]{Panchishkin} 
 A.~Panchishkin, 
`Non-Archimedean L-Funcions of Siegel and Hilbert Modular Forms', 
{\em Lecture Notes in Mathematics}, \textbf{1471}, Springer-Verlag, 
Berlin Heidelberg, (1991).

\bibitem[Sh85]{Shimuraold} 
 G.~Shimura, 
`On Eisenstein Series of Half-integral Weight', 
{\em Duke Mathematical Journal}, \textbf{52} (1985), 281--314.

\bibitem[Sh93]{Shimuratheta}
 G.~Shimura,
`On the Transformation Formulas of Theta Series',
{\em American Journal of Mathematics}, \textbf{115} (5) (1993), 1011--1052.

\bibitem[Sh94]{Shimuraint} 
 G.~Shimura, 
`Euler Products and Fourier Coefficients of Automorphic Forms on 
Symplectic Groups', 
{\em Inventiones Mathematicae}, \textbf{116} (1994), 531--576.

\bibitem[Sh95a]{Shimuraeis} 
 G.~Shimura, 
`Eisenstein Series and Zeta Functions on Symplectic Groups', 
{\em Inventiones Mathematicae}, \textbf{119} (1995), 539--584.

\bibitem[Sh95b]{Shimurahalf} 
 G.~Shimura, 
`Zeta Functions and Eisenstein Series on Metaplectic Groups', 
{\em Inventiones Mathematicae}, \textbf{121} (1995), 21--60.

\bibitem[Sh96]{Shimuraexp} 
 G.~Shimura, 
`Convergence of Zeta Functions on Symplectic and Metaplectic Groups', 
{\em Duke Mathematical Journal}, \textbf{82} (2) (1996), 327--347.

\bibitem[Sh00]{Shimurabook} 
 G.~Shimura, 
`Arithmeticity of Automorphic Forms', 
{\em Mathematical Surveys and Monographs}, \textbf{82}, Amer. Math. Soc. (2000).

\bibitem[St81]{Sturm} 
J.~Sturm, 
`The Critical Values of Zeta Functions Associated to the Symplectic Group', 
{\em Duke Mathematical Journal}, \textbf{48} (2) (1981), 327--350.

\bibitem[Zha74]{Zha}
N.A.~Zharkovskaya, `The Siegel operator and Hecke operators', 
{\em Functional Analysis and its Applications}, \textbf{8} (1974), 113--120.

\bibitem[Zhu84]{Zhu1}
V.G.~Zhuravlev, `Hecke rings for a covering of the symplectic group', 
{\em Mathematics of the USSR-Sbornik}, \textbf{49} (2) (1984), 379--400.

\bibitem[Zhu85]{Zhu2}
V.G.~Zhuravlev, `Euler expansions of theta transforms of Siegel modular forms of half-integral weight and their analytic properties', 
{\em Mathematics of the USSR-Sbornik}, \textbf{51} (1) (1985), 169--191.

\end{thebibliography}
\end{document}